%% file: MSGFEM-Hcurl.tex
\documentclass[onefignum,onetabnum]{siamart190516}

\usepackage{todonotes}
\usepackage{bm}
\usepackage{enumerate}

\input{AharmonicGenEOShared}

\usepackage{graphicx}
\graphicspath{ {./images/} }

\ifpdf
\hypersetup{
  pdftitle={A-harmonic GenEO-Type preconditioner},
  pdfauthor={A. Strehlow}
}
\fi

\theoremstyle{remark}
\theoremstyle{definition}
\theoremstyle{condition}

\newcommand{\rate}{\Lambda}

\usepackage[normalem]{ulem}

\begin{document}

\maketitle
%
\begin{abstract}
This paper addresses the efficient solution of linear systems arising from curl-conforming finite element discretizations of $H(\mathrm{curl})$ elliptic problems with heterogeneous coefficients. We first employ the discrete form of a multiscale spectral generalized finite element method (MS-GFEM) for model reduction and prove that the method exhibits exponential convergence with respect to the number of local degrees of freedom. The proposed method and its convergence analysis are applicable in broad settings, including general heterogeneous ($L^{\infty}$) coefficients, domains and subdomains with nontrivial topology, irregular subdomain geometries, and high-order finite element discretizations. Furthermore, we formulate the method as an iterative solver, yielding a two-level restricted additive Schwarz type preconditioner based on the MS-GFEM coarse space. The GMRES algorithm, applied to the preconditioned system, is shown to converge at a rate of at least $\Lambda$, where $\Lambda$ denotes the error bound of the discrete MS-GFEM approximation. Numerical experiments in both two and three dimensions demonstrate the superior performance of the proposed methods in terms of dimensionality reduction.

\end{abstract}


\section{Introduction}

In this paper, we focus on the efficient solution of the following (positive-definite) curl-curl problem:
\begin{equation}\label{eq:model_problem}
\nabla \times (\nu \nabla \times \bm{u}) + \kappa \bm{u} = \bm{f}, \quad \nu = \nu(\bm{x}) > 0, \quad \kappa = \kappa(\bm{x}) > 0,    
\end{equation}
subject to appropriate boundary conditions. Equation~\cref{eq:model_problem} arises in a variety of electromagnetic applications, including implicit time discretizations of Maxwell’s equations, as well as preconditioning for time-harmonic Maxwell's equations and incompressible magnetohydrodynamics (MHD) models~\cite{chen2007adaptive}. Such problems are commonly referred to as $H(\mathrm{curl})$ elliptic problems. In practical electromagnetic simulations, the coefficients $\nu$ and $\kappa$ in~\cref{eq:model_problem} are often strongly heterogeneous and exhibit high contrast. A representative example is signal integrity analysis in integrated circuit (IC) packages~\cite{shao2010full}, which involve intricate fine-scale structures and materials with widely varying properties. Although standard numerical methods—such as curl-conforming finite elements—are well established for Maxwell-type problems, solving the resulting linear systems efficiently remains a major challenge in the presence of complex and heterogeneous structures. To ensure accurate results, the computational mesh must be sufficiently fine to resolve all fine-scale features, leading to a prohibitively large number of unknowns. Simultaneously, the geometric and material complexities render the linear systems severely ill-conditioned. Therefore, to enable fast and accurate electromagnetic simulations, it is imperative to develop robust and efficient model reduction and iterative solution techniques.

There exists a vast body of literature on multiscale model reduction methods for elliptic multiscale problems~\cite{hou1997multiscale,weinan2003heterognous,efendiev2013generalized,araya2013multiscale,maalqvist2014localization,owhadi2015bayesian,chung2018constraint,hauck2023super}. This list is by no means exhaustive, and we refer readers to recent surveys~\cite{owhadi2019operator,maalqvist2020numerical,altmann2021numerical,chung2023multiscale} for a more comprehensive overview. These methods typically construct a coarse model that captures physically relevant fine-scale features and serves as an efficient approximation of the original model. However, most existing multiscale methods are primarily developed for scalar elliptic equations, with relatively few extensions to Maxwell-type problems. For Maxwell's equations with (locally) periodic coefficients, homogenization-based multiscale methods have been proposed~\cite{bensoussan2011asymptotic,kristensson2003homogenization,cao2010multiscale,henning2016new,verfurth2019heterogeneous,hochbruck2019heterogeneous}. In contrast, due to the inherent complexity of the equations, significantly fewer results exist for Maxwell-type problems with general heterogeneous ($L^{\infty}$) coefficients. In particular, only a limited number of works have addressed multiscale methods for $H(\mathrm{curl})$ elliptic problems in such general settings. For example, in~\cite{chung2019adaptive}, adaptive generalized multiscale finite element methods (GMsFEM)~\cite{efendiev2013generalized} were developed for discretized $H(\mathrm{curl})$ elliptic problems with divergence-free source terms. In~\cite{gallistl2018numerical,henning2020computational}, multiscale methods based on the localized orthogonal decomposition (LOD) technique~\cite{maalqvist2014localization} were proposed for continuous $H(\mathrm{curl})$ problems posed in contractible domains. These approaches critically rely on the availability of a suitable projection operator, which is explicitly known only in the case of natural boundary conditions. Another LOD-based method was presented in~\cite{ren2019homogenization}, though without a rigorous convergence analysis. To the best of our knowledge, multiscale methods with provable convergence for $H(\mathrm{curl})$ elliptic problems under general conditions -- including arbitrary coefficients, source terms, domains, and boundary conditions -- have not yet been reported.

Iterative solvers for \(H(\mathrm{curl})\) elliptic problems have been extensively studied, with a particular emphasis on constructing effective preconditioners. A widely recognized approach in this context is the Hiptmair--Xu preconditioner~\cite{hiptmair2007nodal,kolev2009parallel,hu2021convergence,hu2023convergence}, also known as the auxiliary space method. While this method has achieved considerable success, recent findings~\cite{bootland2023robust} indicate that its robustness can deteriorate in topologically complex domains. Domain decomposition methods (DDMs) form an important class of techniques for designing scalable and robust preconditioners, with the construction of an appropriate coarse space being a critical component. Two-level overlapping Schwarz methods for \(H(\mathrm{curl})\) problems with standard coarse spaces were previously investigated in~\cite{toselli2000overlapping,hiptmair2000overlapping,pasciak2002overlapping}, under the assumption that the computational domain is convex and/or simply connected. Optimal condition number estimates for overlapping Schwarz methods in \(H(\mathrm{curl})\) were recently established in~\cite{liang2024sharp}, also under these geometric assumptions. The techniques developed in~\cite{liang2024sharp} were subsequently extended to domains with nontrivial topology in~\cite{oh2024new}, though only for problems with a non-dominated curl-curl term and using the lowest-order N\'{e}d\'{e}lec elements. It is worth noting that all of the aforementioned DDM results assume constant coefficients. Iterative substructuring methods have also been developed for \(H(\mathrm{curl})\) problems (see, e.g.,~\cite{hu2003nonoverlapping,toselli2006dual,dohrmann2012iterative,calvo2016bddc}), but typically under certain geometric restrictions on the subdomains (e.g., convex polyhedra) and on the regularity of the coefficients. In recent years, spectral coarse spaces~\cite{galvis2010domain,dolean2012analysis,spillane2014abstract,heinlein2019adaptive}, constructed via local generalized eigenvalue problems, have gained significant attention as a means of developing coefficient-robust DDMs. By carefully designing local eigenvalue problems, one can show that the condition number of the resulting preconditioners depends essentially only on a prescribed eigenvalue threshold. However, existing spectral coarse spaces are essentially designed and analyzed for $H^{1}$ elliptic problems, and direct extensions to H(curl) problems are generally not viable. To the best of our knowledge, the only attempt to construct spectral coarse spaces for \(H(\mathrm{curl})\) problems appears in~\cite{bootland2023robust}, where the GenEO spectral coarse space~\cite{spillane2014abstract} is augmented with a near-kernel space composed of gradients of \(H^1\) functions. The resulting preconditioner is shown---both theoretically and numerically---to be robust with respect to coefficient variations and domain topology. A potential limitation of this approach, however, is the large size of the coarse space, stemming from the large near-kernel of the global problem.

The purpose of this paper is to design and analyze efficient multiscale model reduction methods and two-level Schwarz preconditioners for discretized \( H(\mathrm{curl}) \) elliptic problems without restrictive assumptions. The proposed methods are based on the multiscale spectral generalized finite element method (MS-GFEM)—a particular partition of unity finite element method (PUFEM)~\cite{melenk1995generalized} that employs optimal local approximation spaces constructed from local eigenvalue problems. This approach was first introduced by Babu\v{s}ka and Lipton~\cite{babuska2011optimal} and has been systematically developed in recent years; see~\cite{ma2021novel,ma2022error,chupeng2023wavenumber,ma2024exponential,ma2025unified}. In particular, we refer to~\cite{ma2025unified} for a unified theoretical framework with sharpened local convergence rates. For discretized \( H(\mathrm{curl}) \) elliptic problems based on standard curl-conforming finite elements, we develop a discrete MS-GFEM model reduction method following the general theory established in~\cite{ma2025unified}. The proposed method naturally accommodates general boundary conditions and source terms in \( H(\mathrm{curl})' \) without modification. We prove an exponential decay rate for the local approximation errors and also the global error under very general conditions. Specifically, our results apply to general heterogeneous coefficients, domains and subdomains with nontrivial topology, irregular subdomain geometries, and high-order finite element discretizations. Although the methodology aligns with existing applications of MS-GFEM to other elliptic problems, the local convergence analysis in the discretized \( H(\mathrm{curl}) \) setting is substantially more challenging than for \( H^1 \) or continuous \( H(\mathrm{curl}) \) problems, particularly in the general settings considered here. The primary challenges lie in verifying two key ingredients for the local exponential convergence: a discrete Caccioppoli inequality and a weak approximation property. They rely on appropriate Helmholtz decompositions and a detailed analysis of the structure of the underlying N\'{e}d\'{e}lec finite element spaces. A further novel finding of our work is that, in the \( H(\mathrm{curl}) \) context, the MS-GFEM eigenvalues encode meaningful topological information about the associated subdomain. We illustrate this phenomenon with a detailed three-dimensional example; see \cref{3D example}.

Although the discrete MS-GFEM guarantees exponential convergence, the resulting reduced (coarse) model can become prohibitively large for efficient solution of large-scale 3D electromagnetic problems. To mitigate this issue, we reformulate the method as a Richardson iterative solver following \cite{strehlow2024fast}, leading to a two-level restricted additive Schwarz (RAS) preconditioner with the MS-GFEM coarse space. When applied within either the Richardson iteration or the GMRES algorithm, the preconditioner is proved to yield convergence at a rate bounded below by \(\Lambda\), where \(\Lambda\) denotes the global error bound of the underlying MS-GFEM. Notably, the convergence analysis is derived directly from the MS-GFEM framework, rather than relying on the abstract Schwarz theory~\cite{toselli2004domain} typically used in DDMs. The proposed iterative solvers can achieve the desired accuracy with a small number of iterations and a fairly small coarse model, often resulting in superior computational efficiency compared to the direct application of the discrete MS-GFEM. We also note that, unlike the approach in~\cite{bootland2023robust}, the coarse space in our method is constructed purely from local generalized eigenvalue problems, without additional enrichment using near-kernel components.

The rest of this paper is structured as follows. In \cref{sec-2}, we give the problem formulation and introduce some notation and preliminaries that will be used in the subsequent sections. In \cref{sec-3}, we first present the discrete MS-GFEM for H(curl) problems and then derive the two-level Schwarz preconditioner. \Cref{sec-4} is devoted to the proof of the local exponential convergence of the method. Numerical experiments are presented in \cref{sec-5} to assess the performance of the proposed methods.

\section{Problem formulation, notation, and preliminaries}\label{sec-2}
\subsection{Model problem and finite element discretization}
Let $\Omega\subset\mathbb{R}^{d}$, $d=2,3$, be a bounded Lipschitz domain with a polygonal boundary $\Gamma:=\partial \Omega$ ($\Omega$ and $\Gamma$ may have nontrivial topologies). We consider the following curl-curl problem: Find ${\bm u}:\Omega\rightarrow \mathbb{R}^{d}$ such that
\begin{equation}\label{eq:continuous_model_prob}
\left\{
\begin{array}{lll}
 \nabla \times (\nu \nabla\times {\bm u}) + \kappa {\bm u}= {\bm f}\,\quad &{\rm in}\;\, \Omega \\[2mm]
  \qquad \qquad\quad\quad  {\bm n}\times {\bm u} = {\bm 0} \,\quad &{\rm on}\;\, \Gamma, 
\end{array}
\right.
\end{equation}
where ${\bm n}$ is the outward unit normal to $\Gamma$. To simplify the presentation, we only consider the homogeneous Dirichlet boundary condition in \cref{eq:continuous_model_prob}, but the methods presented in this paper are directly applicable to general boundary conditions. Before defining the weak formulation of problem \cref{eq:continuous_model_prob}, we introduce some notation that will be used throughout the paper. For a bounded Lipschitz domain $D\subset\mathbb{R}^{d}$, we define 
\begin{equation}
\begin{split}
{\bm H}({\rm curl};D) :&=\big\{{\bm u}\in {\bm L}^{2}(D): \nabla\times {\bm u}\in {\bm L}^{2}(D) \big\},\\
{\bm H}({\rm div};D) :&=\big\{{\bm u}\in {\bm L}^{2}(D): \nabla\cdot {\bm u}\in L^{2}(D) \big\},\\
{\bm H}_{0}({\rm curl};D) :&=\big\{{\bm u}\in {\bm H}({\rm curl};D): {\bm n}\times {\bm u} = {\bm 0}\;\;\text{on}\;\;\partial D \big\},\\
{\bm H}_{0}({\rm div};D) :&=\big\{{\bm u}\in {\bm H}({\rm div};D): {\bm n}\cdot {\bm u} = 0\;\;\text{on}\;\;\partial D \big\},
\end{split}
\end{equation}
which are equipped with the norms $\Vert {\bm u}\Vert_{{\bm H}({\rm curl};D)} : = \big(\Vert {\bm u}\Vert^{2}_{{\bm L}^{2}(D)} + \Vert \nabla\times {\bm u}\Vert^{2}_{{\bm L}^{2}(D)} \big)^{\frac12}$ and 
$\Vert {\bm u}\Vert_{{\bm H}({\rm div};D)} : = \big(\Vert {\bm u}\Vert^{2}_{{\bm L}^{2}(D)} + \Vert \nabla\cdot {\bm u}\Vert^{2}_{L^{2}(D)} \big)^{\frac12}$, respectively. Moreover, let ${\bm H}_{0}({\rm curl};D)^{\prime}$ denote the dual space of ${\bm H}_{0}({\rm curl};D)$. Throughout this paper, we assume that ${\bm f}\in {\bm H}_{0}({\rm curl};\Omega)^{\prime}$, and that the coefficients $\nu$ and $\kappa$ satisfy the following conditions.
\begin{assumption}\label{ass:coefficients}
\begin{enumerate}
\item[(i)] $\nu\in L^{\infty}(\Omega, \mathbb{R}_{\rm sym}^{d\times d})$, and there exist $0<\nu_{\rm min}\leq \nu_{\rm max}<+\infty$ such that $\nu_{\rm min}|{\bm \xi}|^{2}\leq \nu({\bm x}){\bm \xi}\cdot{\bm \xi}^{T} \leq \nu_{\rm max}|{\bm \xi}|^{2}$ for $a.e \; {\bm x}\in \Omega$ and all ${\bm \xi}\in \mathbb{R}^{d}$.

\vspace{1ex}
\item[(ii)] $\kappa\in L^{\infty}(\Omega, \mathbb{R}_{\rm sym}^{d\times d})$, and there exist $0<\kappa_{\rm min}\leq \kappa_{\rm max}<+\infty$ such that $\kappa_{\rm min}|{\bm \xi}|^{2}\leq \kappa({\bm x}){\bm \xi}\cdot{\bm \xi}^{T} \leq \kappa_{\rm max}|{\bm \xi}|^{2}$ for $a.e \; {\bm x}\in \Omega$ and all ${\bm \xi}\in \mathbb{R}^{d}$.


\end{enumerate}
\end{assumption}
For a domain $D\subset\Omega$, we define a bilinear form $a_{D}:{\bm H}({\rm curl};D)\times {\bm H}({\rm curl};D)\rightarrow \mathbb{R}$ as
\begin{equation}\label{local_sesqui_form}
    a_{D}({\bm v}, {\bm w}):=(\nu\nabla\times{\bm v}, \nabla\times{\bm w})_{{\bm L}^{2}(D)} + (\kappa{\bm v}, {\bm w})_{{\bm L}^{2}(D)}.
\end{equation} 
Moreover, we define the energy norm $\Vert {\bm v}\Vert_{a,D}:=\sqrt{a_{D}({\bm v},{\bm v})}$ for ${\bm v}\in {\bm H}({\rm curl};D)$. When $D=\Omega$, we drop the domain index and simply write $a(\cdot,\cdot)$ and $\Vert {\bm v}\Vert_{a}$. The weak formulation of problem \cref{eq:continuous_model_prob} is defined by: Find ${\bm u}\in {\bm H}_{0}({\rm curl};\Omega)$ such that 
\begin{equation}\label{eq:weak_form}
a({\bm u}, {\bm v}) =  \langle {\bm f},  {\bm v} \rangle_{\Omega}\quad \text{for all}\;\;{\bm v}\in {\bm H}_{0}({\rm curl};\Omega),  
\end{equation}
where $\langle \cdot,\cdot\rangle_{\Omega}$ denotes the duality pairing between ${\bm H}_{0}({\rm curl};\Omega)^{\prime}$ and ${\bm H}_{0}({\rm curl};\Omega)$. By \cref{ass:coefficients}, there exists a constant $C_{\rm ell}>0$, such that for any subdomain $D\subset\Omega$, 
\begin{equation*}
|a_{D}({\bm v}, {\bm v})|\geq C_{\rm ell} \Vert {\bm v}\Vert^{2}_{{\bm H}({\rm curl};D)} \quad \text{for all}\;\;{\bm v}\in {\bm H}({\rm curl};D).  
\end{equation*}
Therefore, by the Lax--Milgram theorem, problem \cref{eq:weak_form} is uniquely solvable.


Now we consider the standard finite element approximation of problem \cref{eq:continuous_model_prob}. Let $\{\mathcal{T}_{h}\}$ be a family of shape-regular simplicial triangulations of $\Omega$, indexed by the mesh size $h:= \max_{K\in \mathcal{T}_{h}} {\rm diam}(K)$. Given an element $K$, let $\mathbb{P}_{r}(K)$ be the set of polynomials of maximum total degree $r$ ($r\in\mathbb{N}$) defined on $K$. The Raviart-Thomas and N\'{e}d\'{e}lec polynomial spaces defined on $K$ are given as
\begin{align}\label{eq:element_polynomials}
\begin{split}
{\bm D}_{r}(K):= \big(\mathbb{P}_{r-1}(K)\big)^{d} \oplus {\bm x}\, \widetilde{\mathbb{P}}_{r-1}(K),\quad {\bm R}_{r}(K):= \big(\mathbb{P}_{r-1}(K)\big)^{d} \oplus \mathcal{\bm S}_{r}(K),
\end{split}
\end{align}
where $\widetilde{\mathbb{P}}_{r-1}(K)$ denotes the space of homogeneous polynomials of degree $r-1$ on $K$, and $\mathcal{\bm S}_{r}(K):=\big\{{\bm p}\in \big(\widetilde{\mathbb{P}}_{r}(K)\big)^{d}: {\bm x}\cdot {\bm p} = 0 \big\}$. On a triangulation $\mathcal{T}_h$, we define the $r$-order $H^{1}(\Omega)$-, ${\bm H}({\rm div};\Omega)$-, and ${\bm H}({\rm curl};\Omega)$-conforming finite element spaces as 
\begin{align}
\begin{split}
S_{h}(\Omega)&:= \big\{v\in H^{1}(\Omega): v|_{K} \in \mathbb{P}_{r}(K)\quad \forall K\in \mathcal{T}_h   \big\},\\    
  {\bm U}_{h}(\Omega)&:= \big\{{\bm v} \in {\bm H}({\rm div};\Omega): {\bm v}|_{K}  \in {\bm D}_r(K)\quad \forall K \in  \mathcal{T}_{h}\big\},\\
  {\bm Q}_{h}(\Omega)&:= \big\{{\bm v} \in {\bm H}({\rm curl};\Omega): {\bm v}|_{K}  \in {\bm R}_r(K)\quad \forall K \in  \mathcal{T}_{h}\big\}, \\
\end{split}
\end{align}
and set $S_{h,0}(\Omega):=S_{h}(\Omega)\cap H_{0}^{1}(\Omega)$, $ {\bm U}_{h,0}(\Omega):=  {\bm U}_{h}(\Omega)\cap {\bm H}_{0}({\rm div};\Omega)$, and ${\bm Q}_{h,0}(\Omega):=  {\bm Q}_{h}(\Omega)\cap {\bm H}_{0}({\rm curl};\Omega)$. 
The standard Galerkin finite element discretization of problem \cref{eq:weak_form} is defined by: Find ${\bm u}_{h}\in {\bm Q}_{h,0}(\Omega)$ such that 
\begin{equation}\label{eq:galerkin_discretization}
a({\bm u}_h, {\bm v}_h) =  \langle {\bm f},  {\bm v}_h \rangle_{\Omega}\quad \text{for all}\;\;{\bm v}_h\in {\bm Q}_{h,0}(\Omega).
\end{equation}
Let $\{{\bm w}_{j}\}_{j=1}^{m}$ be a basis for ${\bm Q}_{h,0}(\Omega)$ with $m:={\rm dim}\,{\bm Q}_{h,0}(\Omega)$. Then, the Galerkin discretization \cref{eq:galerkin_discretization} can be written as a linear system of equations:
\begin{equation}\label{eq:linear_system}
 {\bf Au} = {\bf f}\quad \text{with}\quad {\bf A}\in \mathbb{R}^{m\times m},\quad {\bf A}_{ij}= a({\bm w}_i, {\bm w}_j), \quad \text{and} \quad {\bf f}_i = \langle {\bm f},  {\bm w}_i \rangle_{\Omega}.  
\end{equation}
We conclude this subsection by fixing some notation that will be used in the sequel. For a subdomain $D$ of $\Omega$, we define
\begin{align}\label{eq:local_spaces}
\begin{split}
H^{1}_{\Gamma}(D)&:=\{v\in H^{1}(D): v=0\;\;\, \text{on}\;\;\Gamma\cap \partial D\},  \\
{\bm H}_{\Gamma}({\rm div};D)&:= \{{\bm v}\in {\bm H}({\rm div};D) : {\bm n}\cdot{\bm v}=0\;\;\, \text{on}\;\;\Gamma\cap \partial D\},  \\
{\bm H}_{\Gamma}({\rm curl};D)&:= \{{\bm v}\in {\bm H}({\rm curl};D) : {\bm n}\times{\bm v}={\bm 0}\;\;\, \text{on}\;\;\Gamma\cap \partial D\},\\
S_{h,\Gamma}(D):=\{&v_h|_{D}: v_h\in S_{h,0}(\Omega)\},\; S_{h,0}(D):=S_{h,\Gamma}(D)\cap H_{0}^{1}(D),\\
  {\bm U}_{h,\Gamma}(D):=\{&{\bm v}_h|_{D}: {\bm v}_h\in {\bm U}_{h,0}(\Omega)\},\;  {\bm U}_{h,0}(D) :=  {\bm U}_{h,\Gamma}(D)\cap {\bm H}_{0}({\rm div};D),\\
  {\bm Q}_{h,\Gamma}(D):=\{&{\bm v}_h|_{D}: {\bm v}_h\in {\bm Q}_{h,0}(\Omega)\},\;  {\bm Q}_{h,0}(D): =  {\bm Q}_{h,\Gamma}(D)\cap {\bm H}_{0}({\rm curl};D).
\end{split}
\end{align}
Furthermore, for an element $K\in \mathcal{T}_h$, let $\Sigma_{K}: {\bm H}^{1}(K)\rightarrow {\bm D}_{r}(K)$ and $\Pi_{K}: {\bm H}^{1}(K)\rightarrow {\bm R}_{r}(K)$
be, respectively, the Raviart-Thomas and N\'{e}d\'{e}lec interpolation operators defined by specifying the degrees of freedom (see, e.g., \cite[Chapter 2.5]{boffi2013mixed}). The global interpolation operators $\Sigma_h$ and $\Pi_h$ are defined elementwise as $\Sigma_h{\bm u}|_{K}:= \Sigma_{K} ({\bm u}|_K)$ and $\Pi_h{\bm u}|_{K}:= \Pi_{K} ({\bm u}|_K)$. The operators $\Sigma_h$ and $\Pi_h$ satisfy the following "commuting diagram property":
\begin{equation}\label{eq:commuting_property}
\nabla\times \Pi_h {\bm u} = \Sigma_h \nabla\times {\bm u}\quad \forall {\bm u} \in {\bm H}^{1}(\Omega) \;\text{with}\;\nabla\times {\bm u}\in {\bm H}^{1}(\Omega).    
\end{equation}

\subsection{Preliminaries}\label{subsec:preliminaries}
In this subsection, we present some lemmas and preliminaries that will be used for the analysis in \cref{sec-4}. The first lemma gives the regular decomposition of vector fields in ${\bm H}_{0}({\rm curl};\Omega)$, which can be found in \cite[Theorem 2, Remark 3]{hiptmair2020review} and \cite[Lemma 3.1]{faustmann2022h}. We note that while the proofs in \cite{hiptmair2020review,faustmann2022h} focus on the three-dimensional case, an extension to $\mathbb{R}^{2}$ is straightforward.
\begin{lemma}[Regular decomposition]\label{lem:regular_decomposition}
Let $\Omega\subset \mathbb{R}^{d}$ $(d=2,3)$ be a bounded Lipschitz polyhedral domain. Then for each ${\bm u}\in {\bm H}_{0}({\rm curl};\Omega)$, there exist ${\bm z}\in {\bm H}_{0}^{1}(\Omega)$ and $p\in H_{0}^{1}(\Omega)$ such that ${\bm u} = {\bm z} + \nabla p$, with the estimates
\begin{align}\label{eq:estimate_regular_decomp}
\begin{split}
 \Vert {\bm z}\Vert_{{\bm L}^{2}(\Omega)} +   \Vert  p\Vert_{L^{2}(\Omega)} &\leq C\Vert{\bm u}  \Vert_{{\bm L}^{2}(\Omega)},\\
\Vert \nabla {\bm z}\Vert_{{\bm L}^{2}(\Omega)} + {\rm diam}(\Omega)^{-1} \Vert {\bm z}\Vert_{{\bm L}^{2}(\Omega)}+ \Vert  \nabla p\Vert_{{\bm L}^{2}(\Omega)} &\leq C\Vert \nabla \times {\bm u}  \Vert_{{\bm L}^{2}(\Omega)}, 
\end{split}
\end{align}
where $C>0$ depends only on the shape of $\Omega$, but not on ${\rm diam}(\Omega)$.   
\end{lemma}
The next lemma gives error estimates for the interpolation operator $\Pi_h$ defined above.
\begin{lemma}[{\cite[Theorem 5.41]{monk2003finite}}]\label{lem:error_estimates_interpolation}
Let $\Pi_h$ be the N\'{e}d\'{e}lec interpolation operator, and let $K\in \mathcal{T}_h$ be an element. There exists a constant $C$ independent of $h_K$ such that for $1\leq s\leq r$,
\begin{align}\label{eq:error_nedelec_interpolation_1}
\begin{split}
\Vert {\bm u} -   \Pi_h{\bm u}\Vert_{{\bm L}^{2}(K)}  &\leq Ch_{K}^{s}\big(|{\bm u}|_{{\bm H}^{s}(K)}+  |\nabla\times {\bm u}|_{{\bm H}^{s}(K)}\big),\\
\Vert \nabla\times ({\bm u} -   \Pi_h{\bm u})\Vert_{{\bm L}^{2}(K)}  &\leq Ch_{K}^{s}|\nabla\times {\bm u}|_{{\bm H}^{s}(K)},
\end{split}
\end{align}
where $|\cdot|_{{\bm H}^{1}(K)}$ denotes the ${\bm H}^{1}(K)$ semi-norm. Moreover, if ${\bm u}\in {\bm H}^{1}(K)$ and $\nabla\times {\bm u}|_{K}\in {\bm D}_r(K)$, where ${\bm D}_r(K)$ is defined in \cref{eq:element_polynomials}, then
\begin{equation}\label{eq:error_nedelec_interpolation_2}
\Vert {\bm u} - \Pi_h{\bm u}\Vert_{{\bm L}^{2}(K)}\leq Ch_{K}\big(|{\bm u}|_{{\bm H}^{1}(K)} + \Vert \nabla \times {\bm u}\Vert_{{\bm L}^{2}(K)}\big).
\end{equation}
\end{lemma}

The following lemma shows that the $r$-order derivatives of a polynomial ${\bm p}\in {\bm R}_{r}(K)$ can be bounded by the $(r-1)$-order derivatives of its curl. This property is vital to the proof of a super-approximation estimate for N\'{e}d\'{e}lec elements in \cref{subsec:proof of Cac-inequality}. Whereas it is known for the lowest-order N\'{e}d\'{e}lec element (see, e.g., \cite[Lemma 4.1]{faustmann2022h}), the following result for arbitrary high-order elements seems new. 
\begin{lemma}\label{lem:equiv_norm}
There exists a constant $C$ depending only on $r$ and $d$, such that for any ${\bm p}\in {\bm R}_{r}(K)$, with ${\bm R}_{r}(K)$ defined in \cref{eq:element_polynomials}, 
\begin{equation}\label{eq:norm-control}
    \sum_{|\bm \alpha| = r} |\partial^{\bm \alpha} {\bm p} | \leq C\sum_{|\bm \alpha| = r-1} |\partial^{\bm \alpha} \nabla\times {\bm p} |,
\end{equation}
where ${\bm \alpha}=(\alpha_1,\cdots,\alpha_{d})$ is the multi-index with $\alpha_i\geq 0$, $|{\bm \alpha}| = \alpha_1+\cdots\alpha_d$, and $\partial^{\bm \alpha} = \partial^{\alpha_1}_{x_1}\cdots \partial^{\alpha_d}_{x_d}$. 
\end{lemma}
\begin{proof}
We prove this lemma by a direct computation. Noting that ${\bm R}_{r}(K)= \big(\mathbb{P}_{r-1}(K)\big)^{d} \oplus \mathcal{\bm S}_{r}(K)$, and that for ${\bm p}\in \big(\mathbb{P}_{r-1}(K)\big)^{d}$, $\partial^{\bm \alpha} {\bm p} = {\bm 0}$ for $|{\bm \alpha}| = r$, and $\partial^{\bm \alpha} \nabla\times {\bm p} = {\bm 0}$ for $|{\bm \alpha}| = r-1$, we only need to prove inequality \cref{eq:norm-control} for ${\bm p}\in \mathcal{\bm S}_{r}(K)$. Without loss of generality, we assume that $d=3$. By \cite{bergot2010generation}, the space $\mathcal{\bm S}_{r}(K)$ is spanned by the following polynomials vectors:
\begin{align*}
&\text{for} \;\;0\leq m+n\leq r,\; m\neq 0 \;\;\text{and}\;\;n\neq r: 
\left(
\begin{array}{cc}
x_1^{m-1}x_2^{n}x_3^{r-m-n+1} \\
0\\
-x_1^{m}x_2^{n}x_3^{r-m-n}
\end{array}
      \right),  \\[1ex] 
&\text{for} \;\;0\leq m+n\leq r,\; m\neq r \;\;\text{and}\;\;n\neq 0: 
\left(
\begin{array}{cc}
0\\
x_1^{m}x_2^{n-1}x_3^{r-m-n+1} \\
-x_1^{m}x_2^{n}x_3^{r-m-n}
\end{array}
      \right),  \\[1ex]
 &\text{for} \;\;m+n=r+1,\; m\neq 0 \;\;\text{and}\;\;n\neq 0: 
\left(
\begin{array}{cc}
x_1^{m-1}x_2^{n} \\
-x_1^{m}x_2^{n-1} \\
0
\end{array}
\right).
\end{align*}
For convenience, we denote these three sets of polynomials vectors by ${\bm u}_{m,n}$,  ${\bm v}_{m,n}$, and ${\bm w}_{m,n}$, respectively. Then, any ${\bm p}\in {\bm R}_{r}(K)$ can be written as 
\begin{equation}
    {\bm p} = \sum_{\substack{0\leq m+n\leq r,\\ {m\neq 0, \,n\neq r}}} a_{m,n}{\bm u}_{m,n} + \sum_{\substack{0\leq m+n\leq r,\\ {m\neq r, \,n\neq 0}}} b_{m,n}{\bm v}_{m,n} + \sum_{\substack{ m+n=r+1,\\ {m\neq 0, \,n\neq 0}}} c_{m,n}{\bm w}_{m,n},
\end{equation}
with $a_{m,n},b_{m,n}, c_{m,n}\in \mathbb{R}$. It suffices to prove that 
\begin{equation}\label{eq:reduced_result}
\sum_{|\bm \alpha| = r-1} |\partial^{\bm \alpha} \nabla\times {\bm p} |\geq C\Bigg(\sum_{\substack{0\leq m+n\leq r,\\ {m\neq 0, \,n\neq r}}} |a_{m,n}|  + \sum_{\substack{0\leq m+n\leq r,\\ {m\neq r, \,n\neq 0}}} |b_{m,n}| + \sum_{\substack{ m+n=r+1,\\ {m\neq 0, \,n\neq 0}}} |c_{m,n}|  \Bigg).    
\end{equation}
Noting that 
\begin{align*}
\begin{split}
    \nabla\times {\bm u}_{m,n} = \big(-nx_1^{m}x_2^{n-1}x_3^{r-m-n}, \;& (r-n+1)x_1^{m-1}x_2^{n}x_3^{r-m-n},\\
        -n x_1^{m-1}x_2^{n-1}x_3^{r-m-n+1} \big),  \;\;\nabla\times {\bm v}_{m,n} =& \big(-(r-m+1)x_1^{m}x_2^{n-1}x_3^{r-m-n}, \\
          mx_1^{m-1}x_2^{n}x_3^{r-m-n}, mx_1^{m-1}x_2^{n-1}x_3^{r-m-n+1} \big), & \nabla\times {\bm w}_{m,n} = \big(0, 0, -(m+n) x_1^{m-1}x_2^{n-1}\big),
\end{split}
\end{align*}
we find that for any $(m>1, n>1, m+n<r)$, the sum $\sum_{|\bm \alpha| = r-1} |\partial^{\bm \alpha} \nabla\times {\bm p}|^{2}$ contains the following two terms involving $a_{m,n}$ and $b_{m,n}$:
\begin{align*}
\begin{split}
    [m(n-1)(r-m-n)]^{2} \big(na_{m,n} &+ (r-m+1)b_{m,n}\big)^{2} \\
    + [(m-1)n(r-m-n)]^{2} \big((r-n+1)a_{m,n} &+ mb_{m,n}\big)^{2}  \geq C(|a_{m,n}|^{2} + |b_{m,n}|^{2}),
\end{split}
\end{align*}
where the inequality above follows from the fact that $m+n\neq r+1$. Similar estimates hold for other coefficients $(a_{m,n},b_{m,n})$ and also $c_{m,n}$. Therefore, inequality \cref{eq:reduced_result} holds true and the lemma is thus proved.
\end{proof}

Next we discuss the spaces of "harmonic forms". For a domain $D\subset \Omega$ which is a union of elements in $\mathcal{T}_h$ and may intersect $\Gamma$, we recall the spaces defined in \cref{eq:local_spaces} and consider the de Rham complex with \textit{partial boundary conditions}:
\begin{equation}\label{eq:continuous_deRham_complex} 
H^{1}_{\Gamma}(D) \xrightarrow{\;\; \nabla\;\;} {\bm H}_{\Gamma}({\rm curl};D) \xrightarrow{\;\; \nabla\times \;\;} {\bm H}_{\Gamma}({\rm div};D) \xrightarrow{\;\; \nabla\cdot \;\;} L^{2}(D).
\end{equation}
The corresponding discrete de Rham complex reads
\begin{equation}\label{eq:discrete_deRham_complex} 
S_{h,\Gamma}(D) \xrightarrow{\;\; \nabla\;\;} {\bm Q}_{h,\Gamma}(D) \xrightarrow{\;\; \nabla\times \;\;} {\bm U}_{h,\Gamma}(D) \xrightarrow{\;\; \nabla\cdot \;\;} Z_h(D),
\end{equation}
where $Z_h(D)\subset L^{2}(D)$ denotes the space consisting of piecewise polynomials of maximum total degree $r-1$. Furthermore, we consider the space of harmonic 1-forms (the first cohomology space) associated with the de Rham complex with the $\kappa$-weighted $L^{2}$-norm:
\begin{align}\label{eq:continuous-harmonic-forms}
\begin{split}
\mathcal{\bm H}(D):= \big\{{\bm v}\in {\bm H}_{\Gamma}({\rm curl};D): \nabla\times {\bm v} = {\bm 0},\;(\kappa{\bm v}, \nabla \xi)_{{\bm L}^{2}(D)} = 0 \;\, \forall\,\xi\in H^{1}_{\Gamma}(D)\big\},       
\end{split}
\end{align}
and its discrete counterpart:
\begin{align}\label{eq:general_discrete-harmonic-forms}
\begin{split}
\mathcal{\bm H}_{h}(D):= \big\{{\bm v}_h\in \,&{\bm Q}_{h,\Gamma}(D): \nabla\times {\bm v}_h = {\bm 0},\;\\
(\kappa{\bm v}_h, \nabla \xi_h)_{{\bm L}^{2}(D)} &= 0 \;\, \forall\,\xi_h\in S_{h,\Gamma}(D)\big\}.       
\end{split}
\end{align}
Whereas the de Rham complexes and the spaces of harmonic forms in the case of no boundary conditions ($\partial D\cap \Gamma = \emptyset$) or of full boundary conditions ($\partial D= \Gamma$) have been well studied (see \cite{arnold2006finite,arnold2010finite,arnold2018finite}), the case of partial boundary conditions is much less investigated. Of particular importance to our analysis is the dimension of $\mathcal{\bm H}_{h}(D)$. Indeed, ${\rm dim}\,\mathcal{\bm H}_{h}(D)$ depends only on $D$, but not on the underlying triangulation, as shown by the following lemma. We also note that ${\rm dim}\,\mathcal{\bm H}_{h}(D)$ is independent of the coefficient $\kappa$ that satisfies \cref{ass:coefficients}(ii); see \cite[section 6.1]{arnold2010finite}.

\begin{lemma}\label{lem:dimension_harmonic_forms}
Let $D$ be a subdomain of $\Omega$ which is a union of elements in $\mathcal{T}_h$. Then, ${\rm dim}\,\mathcal{\bm H}_{h}(D) = {\rm dim}\,\mathcal{\bm H}(D)$.
\end{lemma}
\begin{proof}
We show this result by following the proofs in \cite{arnold2018finite,arnold2006finite} for the case of no boundary conditions. In fact, there are two ways to prove the result. The first is based on de Rham's theorem \cite[Theorem 2.2]{arnold2018finite} and de Rham's mappings. The second way is to use \cite[Theorem 5.1]{arnold2018finite}, which shows the isomorphism of cohomology between the continuous and discrete complexes under the assumption of the existence of a bounded cochain projection. The required projection has been constructed recently in \cite{licht2019smoothed} for the case of partial boundary conditions (see also \cite{hiptmair2019discrete}). 
\end{proof}

Next we discuss the dimension of $\mathcal{\bm H}(D)$. When $\partial D\cap \Gamma \neq \emptyset$, ${\rm dim}\,\mathcal{\bm H}(D)$ depends not only on the topology of $D$, but also on the connectedness of $\partial D\cap \Gamma$. Let $b_{1}(D,\Gamma\cap \partial D)$ denote the first \textit{relative} Betti number of the pair ($D$, $\Gamma\cap \partial D$), which is defined as the dimension of the first singular homology group of $D$ relative to $\Gamma\cap \partial D$; see \cite[Chapter 4, Section 4]{spanier1989algebraic} for details. By \cite[Theorem 5.3]{gol2011hodge}, we have the following result.
\begin{lemma}\label{lem:Betti-numbers}
Let $D$ be a subdomain of $\Omega$ which is a union of elements in $\mathcal{T}_h$. Then, ${\rm dim}\,\mathcal{\bm H}(D) = b_{1}(D, \Gamma\cap \partial D)$.
\end{lemma}
\begin{remark}
If $\Gamma\cap \partial D = \emptyset$, then $b_{1}(D, \Gamma\cap \partial D)$ reduces to the classical first Betti number of $D$, which is the number of holes (or `tunnels' in 3D) through the domain. In general, it depends on topological properties of $D$ and $\Gamma\cap \partial D$. For example, if $D = (0,1)^{3}$ and $\Gamma\cap \partial D$ has $M$ connected components, then $b_{1}(D, \Gamma\cap \partial D) = M-1$. We refer to \cite[Propositions 3.8 and 3.13]{amrouche2021vector} for its value for more complicated domains.    
\end{remark}


\section{MS-GFEM: model reduction and preconditioning}\label{sec-3}
\subsection{Discrete MS-GFEM}
In this subsection, we present the discrete MS-GFEM as a model reduction method for the discrete problem \cref{eq:galerkin_discretization} following \cite{ma2022error,ma2025unified}. We first briefly describe the construction of a "discrete GFEM" in the finite element setting. Let $\{\omega_i\}_{i=1}^{M}$ be an overlapping decomposition of $\Omega$ with each subdomain $\omega_i$ consisting of a union of mesh elements in $\mathcal{T}_{h}$. We also introduce a partition of unity $\{\chi_i\}_{i=1}^{M}$ subordinate to $\{\omega_i\}_{i=1}^{M}$ that satisfies for all $i$, $0\leq \chi_i\leq 1$, ${\rm supp}\,\chi_i\subset \overline{\omega_{i}}$, $\chi_i\in W^{1,\infty} (\omega_i)$, and $\sum_{i=1}^{M}\chi_i = 1$. For each $i=1,\cdots,M$, let $\Xi_{h,i}: {\bm Q}_{h,\Gamma}(\omega_i)\rightarrow {\bm Q}_{h,0}(\omega_i)$ be the operator defined by
\begin{equation}\label{eq:discrete_pou}
 \Xi_{h,i}({\bm v}_h) = \Pi_{h}(\chi_i {\bm v}_h),   
\end{equation}
where $\Pi_{h}$ denotes the N\'{e}d\'{e}lec interpolation operator, and the spaces ${\bm Q}_{h,\Gamma}(\omega_i)$ and ${\bm Q}_{h,0}(\omega_i)$ are defined following \cref{{eq:local_spaces}}. Note that the interpolant $\Pi_{h}(\chi_i {\bm v}_h)$ is well defined. The operators $\{\Xi_{h,i}\}_{i=1}^{M}$ form a \textit{discrete} partition of unity, in the sense that
\[
\sum_{i=1}^{M}R_{i}^{\top}\Xi_{h,i}({\bm v}_h|_{\omega_i}) = {\bm v}_h\quad \forall {\bm v}_h\in {\bm Q}_{h,0}(\Omega),
\]
where $R_i^\top : {\bm Q}_{h,0}(\omega_i) \rightarrow {\bm Q}_{h,0}(\Omega)$ denotes the zero extension. 

For each subdomain $\omega_i$, let a local particular function ${\bm u}^{p}_{i}\in {\bm Q}_{h,0}(\omega_i)$ and a local approximation space ${\bm S}_{n_i}(\omega_i)\subset  {\bm Q}_{h,\Gamma}(\omega_i)$ of dimension $n_i$ be given. The global particular function and global approximation space of GFEM are defined as
\begin{equation}
{\bm u}^{p} := \sum_{i=1}^{M} R_i^\top\Xi_{h,i} ({\bm u}^{p}_{i}),\quad  {\bm S}_{n}(\Omega) := {\rm span}\Big\{\sum_{i=1}^{M} R_i^\top\Xi_{h,i} ({\bm v}_{i}): {\bm v}_{i}\in  {\bm S}_{n_i}(\omega_i)\Big\}. 
\end{equation}
By construction, we have ${\bm u}^{p}\in {\bm Q}_{h,0}(\Omega)$ and ${\bm S}_{n}(\Omega)\subset {\bm Q}_{h,0}(\Omega)$. The GFEM approximation of the fine-scale solution ${\bm u}_h$ of \cref{eq:galerkin_discretization} is defined by finding ${\bm u}^{G} = {\bm u}^{p} + {\bm u}^{s}$, with ${\bm u}^{s}\in  {\bm S}_{n}(\Omega)$, such that
\begin{equation}
a({\bm u}^{G}, {\bm v}) =  \langle {\bm f},  {\bm v} \rangle_{\Omega}\quad \text{for all}\;\;{\bm v}\in {\bm S}_{n}(\Omega).    
\end{equation}

In the following, we construct the local particular functions ${\bm u}^{p}_{i}$ and the local approximation spaces $ {\bm S}_{n_i}(\omega_i)$ within the framework of MS-GFEM. To do this, we introduce an oversampling domain $\omega_i^{\ast}\supset \omega_i$ for each subdomain $\omega_i$, formed by adding a few layers of adjoining fine mesh elements to $\omega_i$. 
On each $\omega^{\ast}_i$, we consider the following local variational problem: Find ${\bm \phi}_{h,i}\in {\bm Q}_{h,0}(\omega^{\ast}_i)$ such that 
\begin{equation}\label{eq:local_BVP}
a_{\omega_i^{\ast}}({\bm \phi}_{h,i}, {\bm v}_h) =  \langle {\bm f}, \widetilde{R}_{i}^\top{\bm v}_h \rangle_{\Omega}\quad \text{for all}\;\;{\bm v}_h\in {\bm Q}_{h,0}(\omega_i^{\ast}),
\end{equation}
where $\widetilde{R}_{i}^\top : {\bm Q}_{h,0}(\omega^{\ast}_i) \rightarrow {\bm Q}_{h,0}(\Omega)$ denotes the zero extension. By \cref{ass:coefficients}, problem \cref{eq:local_BVP} is uniquely solvable. Combining \cref{eq:galerkin_discretization,eq:local_BVP}, we see that ${\bm u}_{h}|_{\omega_i^{\ast}} - {\bm \phi}_{h,i}$ lies in the following discrete harmonic space:
\begin{equation}\label{eq:harmonic-space}
{\bm W}_{h}(\omega_i^{\ast}) := \big\{{\bm v}_{h}\in {\bm Q}_{h,\Gamma}(\omega^{\ast}_i): a_{\omega_i^{\ast}} ({\bm v}_{h}, {\bm w}_{h}) =0 \quad\text{for all}\;\; {\bm w}_{h}\in {\bm Q}_{h,0}(\omega_i^{\ast}) \big\}.    
\end{equation}
Now we want to approximate ${\bm u}_{h}|_{\omega_i^{\ast}} - {\bm \phi}_{h,i}$ in $\omega_i$. Here a key observation is that ${\bm W}_{h}(\omega_i^{\ast})|_{\omega_i}$ is essentially low-dimensional which will be made clear. To identify its low-dimensional approximation, we use the theory of the Kolmogorov $n$-width \cite{pinkus1985n}. Let $P_{i}: \big({\bm W}_{h}(\omega_i^{\ast}), \Vert\cdot\Vert_{a,\omega_i^{\ast}}\big) \rightarrow \big({\bm Q}_{h,0}(\omega_i),\Vert\cdot\Vert_{a,\omega_i}\big)$ be the operator defined by
\begin{equation}
P_{i} {\bm v}_{h} = \Xi_{h,i}({\bm v}_{h}|_{\omega_i}). 
\end{equation}
We consider the following best approximation problem associated with $P_i$:
\begin{equation}\label{eq:n-width}
    d_{n}(P_i) := \inf_{{\bm X}(n)\subset {\bm Q}_{h,0}(\omega_i)} \sup_{{\bm z}_{h}\in {\bm W}_{h}(\omega_i^{\ast}) }\inf_{{\bm v}_{h}\in {\bm X}(n)} \frac{\big\Vert P_{i} {\bm z}_h - {\bm v}_{h}   \big\Vert_{a,\omega_i}} {\Vert {\bm z}_h\Vert_{a,\omega_i^{\ast}}}.
\end{equation}
 The quantity $d_{n}(P_i)$ is known as the Kolmogorov $n$-width of operator $P_{i}$; see \cite[Chapter 2, Section 7]{pinkus1985n}. The $n$-width $d_{n}(P_i)$ can be characterized by means of the singular value decomposition (SVD) of operator $P_i$. Let $P^{\ast}_{i}$ denote the adjoint of $P_{i}$. We consider the following eigenvalue problem: Find $\lambda \in \mathbb{R}$, ${\bm \psi}_{h}\in {\bm W}_{h}(\omega_i^{\ast})$ such that 
\begin{equation}\label{eq:eigenproblem_operator}
    P^{\ast}_{i}P_{i} {\bm \psi}_{h} = \lambda {\bm \psi}_{h},
\end{equation}
which can be written in variational form as
\begin{equation}\label{eq:eigenproblem_variational}
    a_{\omega_i}\big(\Xi_{h,i}({\bm \psi}_{h}|_{\omega_i}),\,\Xi_{h,i}({\bm v}_{h}|_{\omega_i}) \big)  = \lambda a_{\omega^{\ast}_i}({\bm \psi}_{h},{\bm v}_{h})\quad \forall {\bm v}_{h}\in {\bm W}_{h}(\omega_i^{\ast}).
\end{equation}
We have the following characterization of $d_{n}(P_i)$ \cite[Theorem 2.5]{pinkus1985n}.
\begin{lemma}
For each $j=1,\cdots,{\rm dim}\,{\bm W}_{h}(\omega_i^{\ast})$, let $(\lambda^{i}_{j},{\bm \psi}^{i}_{h,j})\in \mathbb{R}\times {\bm W}_{h}(\omega_i^{\ast})$ be the $j$-th eigenpair (with eigenvalues sorted in decreasing order) of problem \cref{eq:eigenproblem_variational}. Then, $d_{n}(P_i) = \sqrt{\lambda^{i}_{n+1}}$, and the optimal approximation space is given by ${\bm X}^{\rm opt}(n):={\rm span}\big\{P_{i}{\bm \psi}^{i}_{h,1},\cdots, P_{i}{\bm \psi}^{i}_{h,n} \big\} $.   
\end{lemma}

Now we can define the local particular functions ${\bm u}^{p}_{i}$ and local approximation spaces $ {\bm S}_{n_i}(\omega_i)$ for MS-GFEM.
\begin{lemma}\label{lem:local_approximation_error}
On each subdomain $\omega_i$, let the local particular function ${\bm u}^{p}_{i}$ and local approximation space $ {\bm S}_{n_i}(\omega_i)$ be defined as
    \begin{equation}\label{eq:local_approximations}
{\bm u}^{p}_{i} := {\bm \phi}_{h,i}|_{\omega_i} ,\quad {\bm S}_{n_i}(\omega_i) := {\rm span}\{{\bm \psi}^{i}_{h,1}|_{\omega_i},\cdots, {\bm \psi}^{i}_{h,n_i}|_{\omega_i} \},  
\end{equation}
where ${\bm \phi}_{h,i}$ and $\{{\bm \psi}^{i}_{h,j}\}$ are the solution of problem \cref{eq:local_BVP} and eigenfunctions of problem \cref{eq:eigenproblem_variational}, respectively. Then, 
\begin{equation}\label{eq:local_estimate}
    \inf_{{\bm v}\in {\bm u}^{p}_{i} + {\bm S}_{n_i}(\omega_i)} \big\Vert \Xi_i({\bm u}_h|_{\omega_i} - {\bm v}) \big\Vert_{a,\omega_i}\leq d_{n_i+1}(P_i) \Vert {\bm u}_{h}\Vert_{a,\omega_i^{\ast}}, 
\end{equation}
where ${\bm u}_h$ is the solution of \cref{eq:galerkin_discretization}, and $d_{n_i+1}(P_i)$ is the $n$-width defined in \cref{eq:n-width}.
\end{lemma}
\begin{proof}
Estimate \cref{eq:local_estimate} follows from the definition of $n$-width, the fact that ${\bm u}_{h}|_{\omega_i^{\ast}} - {\bm \phi}_{h,i}\in {\bm W}_h(\omega_i^{\ast})$, and the estimate $\Vert {\bm u}_{h}|_{\omega_i^{\ast}} - {\bm \phi}_{h,i}\Vert_{a,\omega_i^{\ast}} \leq  \Vert {\bm u}_{h}\Vert_{a,\omega_i^{\ast}}$.   
\end{proof}

Before giving the global error estimate for MS-GFEM, we define
\begin{equation}\label{coloring-constant}
k_{0}:= \max_{x \in \Omega}{\left( \text{card}\{i \ | \ x \in \omega_i \} \right)}, \quad \text{and}\quad k_{0}^{\ast}:= \max_{x \in \Omega}{\left( \text{card}\{i \ | \ x \in \omega^{\ast}_i \} \right)},
\end{equation}
i.e., $k_0$ ($k_{0}^{\ast}$) is the maximum number of subdomains (resp. oversampling domains) that can overlap at any given point. The following theorem shows that the global error of MS-GFEM is essentially bounded by the maximum local approximation error. It can be proved by combining C\'{e}a's Lemma, \cref{lem:local_approximation_error}, and a general approximation theorem of GFEM; see \cite[Theorem 3.20]{ma2025unified} for its proof in an abstract setting. 
\begin{theorem}\label{thm:global_error_bound}
Let ${\bm u}_h$ be the solution of \cref{eq:galerkin_discretization}, and ${\bm u}^{G}$ be the GFEM approximation with the local particular functions and local approximation spaces defined in \cref{eq:local_approximations}. Then,
\begin{equation}
\Vert {\bm u}_h -   {\bm u}^{G}\Vert_{a} \leq \Lambda \Vert    {\bm u}_h \Vert_{a},
\end{equation}
where 
\begin{equation}\label{error_bound}
    \Lambda:= \Big( k_{0}k_{0}^{\ast} \max_{1\leq i\leq M}\lambda^{i}_{n_i+1}\Big)^{1/2} = \sqrt{k_{0}k_{0}^{\ast}} \max_{1\leq i\leq M}d_{n_i+1}(P_i)
\end{equation}
denotes the global error bound of MS-GFEM, with the $n$-width $d_{n}(P_i)$ defined in \cref{eq:n-width}.
\end{theorem}

\subsection{Iterative MS-GFEM and two-level RAS preconditioner}
In this subsection, we present the iterative MS-GFEM algorithm following \cite{strehlow2024fast}, which yields a two-level RAS preconditioner for the linear system \cref{eq:linear_system}. Before defining the algorithm, we note that the domain decomposition $\{\omega_i\}$, the oversampling domains $\{\omega^{\ast}_i\}$, and the coarse space ${\bm S}_{n}(\Omega)$ are fixed in this subsection. Let $\pi_j: {\bm Q}_{h,0}(\Omega)\rightarrow {\bm Q}_{h,0}(\omega_j^*)$ and $\pi_H: {\bm Q}_{h,0}(\Omega)\rightarrow {\bm S}_{n}(\Omega)$ be the $a$-orthogonal projections, i.e., for any ${\bm v}\in {\bm Q}_{h,0}(\Omega)$,
\begin{equation}
 \begin{split}
a_{\omega_j^{\ast}}(\pi_j({\bm v}),{\bm w}) &= a({\bm v}, \widetilde{R}_{j}^\top {\bm w}) \quad \forall {\bm w} \in {\bm Q}_{h,0}(\omega_j^*),\\[1mm]
a(\pi_H({\bm v}),{\bm w}) &= a({\bm v}, {\bm w}) \quad \quad \;\;\forall {\bm w} \in {\bm S}_{n}(\Omega).
\end{split}   
\end{equation}
Here $\widetilde{R}_{j}^\top: {\bm Q}_{h,0}(\omega_j^{\ast})\rightarrow {\bm Q}_{h,0}(\Omega)$ denotes the zero extension. Similarly, we let $R_{j}^\top: {\bm Q}_{h,0}(\omega_j)\rightarrow {\bm Q}_{h,0}(\Omega)$ be the corresponding zero extension. Moreover, we define the operators $\widetilde{\Xi}_{j}: {\bm Q}_{h,0}(\omega_j^*)\rightarrow {\bm Q}_{h,0}(\omega_j)$ ($1\leq j\leq M$) by $\widetilde{\Xi}_{j}({\bm v}) = {\Xi}_{h,j}({\bm v}|_{\omega_j})$.
With these operators, we can define the following MS-GFEM map $G$:
    $$
     G: \; {\bm Q}_{h,0}(\Omega) \rightarrow  {\bm Q}_{h,0}(\Omega),\quad    G({\bm v}) := \sum\limits_{j=1}^M R_{j}^{\top}\widetilde{\Xi}_{j} \pi_j({\bm v}) + \pi_H \Big( {\bm v} - \sum\limits_{j=1}^M R_{j}^{\top}\widetilde{\Xi}_{j} \pi_j({\bm v}) \Big).
    $$
Note that $G({\bm v})$ is the MS-GFEM approximation for the fine-scale FE problem \cref{eq:galerkin_discretization} with the right-hand side $\langle {\bm f},  \cdot \rangle_{\Omega}: = a({\bm v},\cdot)$.     
Now we can define the iterative MS-GFEM algorithm using the operator $G$. 
\begin{definition}[Iterative MS-GFEM] \label{iteration}
Let ${\bm u}_{h}$ be the solution of problem \cref{eq:galerkin_discretization}. Given an initial guess ${\bm u}^0 \in{\bm Q}_{h,0}(\Omega)$, let the sequence $\{{\bm u}^j\}_{j \in \mathbb{N}}$ be defined by
    \begin{equation} \label{richardson-iteration-equation}
        {\bm u}^{j+1} := {\bm u}^j + G\left({\bm u}_h - {\bm u}^j \right),\quad  j=0,1\cdots.
    \end{equation}
\end{definition}
The $(j+1)$-th iterate ${\bm u}^{j+1}$ indeed corresponds to the MS-GFEM approximation with the modified boundary condition ${\bm n}\times {\bm \phi}_{h,i} = {\bm n}\times {\bm u}^{j}$ on $\partial \omega_i^{\ast}$ for the local particular functions ${\bm \phi}_{h,i}$. The following lemma shows that the iterative MS-GFEM algorithm converges at a rate of at least $\Lambda$, where $\Lambda$ is the error bound of MS-GFEM defined by \cref{error_bound}. The result is a consequence of \cref{thm:global_error_bound}; see also \cite[Proposition 4.10]{strehlow2024fast}.
\begin{lemma}
    Let ${\bm u}_{h}$ be the solution of problem \cref{eq:galerkin_discretization}, and let the iterates $\{{\bm u}^{j}\}_{j \in \mathbb{N}}$ be generated by \cref{richardson-iteration-equation}. Then, for $j=1,2,\cdots$,
    $$
        \lVert {\bm u}^{j+1} - {\bm u}_{h} \rVert_a \leq \Lambda \lVert{\bm u}^{j} - {\bm u}_{h} \rVert_{a},
    $$  
where $\Lambda$ is defined by \cref{error_bound}.     
\end{lemma}

Next we give the matrix representation of the iterative MS-GFEM algorithm. Let ${\bf R}_j^T$ and $\widetilde{\bf R}_j^T$ $(1\leq j\leq M)$ be the matrix representations of the zero extensions $R_{j}^{\top}$ and $\widetilde{R}_{j}^{\top}$ with respect to the chosen basis of ${\bm Q}_{h,0}(\Omega)$, respectively. Similarly, we denote by ${\bf R}^{T}_H$ and $\widetilde{\bm \Xi}_j$ the matrix representations of the natural embedding ${\bm S}_{n}(\Omega) \rightarrow {\bm Q}_{h,0}(\Omega)$ and of the operator $\widetilde{\Xi}_{j}$, respectively. Recalling the matrix ${\bf A}$ defined in \cref{eq:linear_system}, let
\begin{equation}
\widetilde{\bf A}_j := \widetilde{\bf R}_j {\bf A} \widetilde{\bf R}_j^T, \quad {\bf A}_H := {\bf R}_H {\bf A} {\bf R}_H^T.
\end{equation}
Then the MS-GFEM map $G$ can be written in matrix form as ${\bf G}= {\bf BA}$, with
\begin{equation} \label{preconditioner}
    {\bf B} := \Big(  \sum_{j=1}^{M} {\bf R}_j^T \widetilde{\bm \Xi}_j\widetilde{\bf A}_j^{-1} \widetilde{\bf R}_j \Big)
    + 
    \big({\bf R}_H^T {\bf A}_H^{-1} {\bf R}_H \big) 
    \Big( {\bf I} - {\bf A}\sum_{j=1}^{M} {\bf R}_j^T \widetilde{\bm \Xi}_j\widetilde{\bf A}_j^{-1} \widetilde{\bf R}_j\Big).
\end{equation}    
Hence, the iteration \cref{richardson-iteration-equation} can be written in matrix form as 
\begin{equation}\label{eq:mat_iterative_MSGFEM}
    {\bf u}^{j+1}:= {\bf u}^{j} + {\bf B}{\bf A}({\bf u} -  {\bf u}^{j}) =  {\bf u}^{j} + {\bf B}({\bf f} -  {\bf A}{\bf u}^{j}),
\end{equation}
where ${\bf B}$ is the desired two-level RAS preconditioner, built by adding the MS-GFEM coarse space to the one-level RAS preconditioner multiplicatively. Note that this preconditioner formally corresponds to the adapted deflation preconditioner `A-DEF2' introduced in \cite{tang2009comparison}. In practice, it is often used in conjunction with a Krylov subspace method to accelerate convergence. Since ${\bf B}$ is non-symmetric, we apply GMRES to solve the preconditioned linear system ${\bf BAu} = {\bf Bf}$. To derive the convergence rate for GMRES, we assume that the inner product $b(\cdot,\cdot)$ used in GMRES satisfies
\begin{equation*}
 b_{1} \sqrt{b({\bf v}, {\bf v})}\leq \sqrt{{\bf v}^{T}{\bf A}{\bf v}} \leq b_{2} \sqrt{b({\bf v}, {\bf v})} \quad \forall {\bf v}\in \mathbb{R}^{m}
\end{equation*}
for some $b_{1},b_{2}>0$. In particular, if we use the energy inner product in GMRES, then $b_1=b_2=1$. Let $\{{\bf u}^{j} \}_{j\in\mathbb{N}}$ be the GMRES iterates with an initial vector ${\bf u}^{0}$. The following convergence estimate was proved in \cite[section 3.2]{strehlow2024fast}.
\begin{theorem}\label{thm:gmresconvergence}
Let $\Lambda$ defined by \cref{error_bound} satisfy $\Lambda<1$. Then the norm of the residual achieved at the $j$th step of the GMRES algorithm satisfies 
    $$
        \lVert {\bf B}({\bf f} - {\bf Au}^{j}) \rVert_b \leq \rate^j \left( \frac{1+\rate}{1-\rate}\frac{b_2}{b_1}  \right) \lVert {\bf B}({\bf f} - {\bf Au}^{0}) \rVert_b.
    $$
\end{theorem} 

\Cref{thm:gmresconvergence} shows that the preconditioned GMRES converges at least as fast as the iterative MS-GFEM algorithm -- at a rate of $\Lambda$. By the definition of $\Lambda$, we can control the convergence rate by preselecting an eigenvalue tolerance \texttt{tol} and then choosing the number $n_i$ of local eigenfunctions such that $\sqrt{\lambda^{i}_{n_i+1}}\leq \texttt{tol}< \sqrt{\lambda^{i}_{n_i}}$. Importantly, the exponential decay of $\sqrt{\lambda^{i}_{j}}$ (with respect to $j$) proved in the next section guarantees that it does not need a lot of local basis functions for a reasonably fast convergence.  

\section{Local exponential convergence}\label{sec-4}
We have seen from \cref{thm:global_error_bound} that the global error of MS-GFEM is essentially bounded by the maximum local approximation error, which, in turn, controls the convergence rates of the iterative MS-GFEM and the preconditioned GMRES algorithms. Therefore, it is critical to derive an explicit decay rate for the local approximation errors, or equivalently, for the $n$-width $d_{n}(P_i)$ in terms of $n$. For ease of presentation, we first consider the special case where $\omega_i$ and $\omega_i^{\ast}$ are simply connected domains with $\omega_i^{\ast}\subset\subset \Omega$. More general situations will be treated in \cref{subsec:nontrivial_topology}. Our result needs two technical assumptions. For a domain $D\subset\omega_i^{\ast}$, let  
\begin{equation}
\mathcal{T}_{h}(D):=\{ K\in \mathcal{T}_h: |K\cap D|>0\}, \quad D^{+}:= {\rm int} \Big( \bigcup_{K\in \mathcal{T}_{h}(D)} \overline{K}\Big),     
\end{equation}
and we define a set of domains 
\begin{align}\label{eq:set_of_domains}
    \begin{split}
 \mathscr{E}_h(D) := \big\{  \widetilde{D}: D^{+}\subset &\widetilde{D}\subset \overline{\omega_i^{\ast}}, \;\widetilde{D}\; \text{is a union of elements in}\;\mathcal{T}_h,\\
& {\rm dist}(\partial \widetilde{D}\setminus \partial \Omega, \partial D\setminus \partial\Omega) \leq 2h\big\}.
    \end{split}
\end{align}
\begin{assumption}\label{ass:simply-connected}
For a simply connected domain $D\subset\omega_i^{\ast}$, there exists a simply connected domain $\widetilde{D}\in  \mathscr{E}_h(D)$.
\end{assumption}
Roughly speaking, \cref{ass:simply-connected} means that given a simply connected domain $D$, we can find a slightly larger, simply connected domain by adding one or two layers of adjoining elements to $D$. This is generally doable by filling the `holes'. The second assumption is on the regularity of the partition of unity $\{\chi_i\}_{i=1}^{M}$. We need a higher regularity than it has to be to estimate the norm of the operators $\Xi_{h,i}$ defined in \cref{eq:discrete_pou}.
\begin{assumption}\label{ass:POU}
For each $i=1,\cdots,M$, $\chi_i\in W^{2,\infty}(\omega_i)$, and there exist $C>0$ depending only on $d$ and $\rho_i>0$ such that $|\chi_i|_{W^{j,\infty}(\omega_i)}\leq C\rho^{-j}_i$ for $j=1,2$.    
\end{assumption}
We note that $\rho_i$ above is related to the size of overlap.

The following theorem is the main result of this paper that gives an exponential decay estimate for the $n$-width $d_{n}(P_i)$. For brevity, we omit the subdomain index, and denote by $\Vert \Xi_{h}\Vert$ the norm of the operator $\Xi_{h}: ({\bm Q}_{h,\Gamma}(\omega),\,\Vert\cdot\Vert_{a,\omega})\rightarrow ({\bm Q}_{h,0}(\omega),\,\Vert\cdot\Vert_{a,\omega})$.

\begin{theorem}\label{thm:exponential_convergence}
Let $\omega\subset \omega^{\ast}\subset\subset\Omega$ with $\delta^{\ast}:={\rm dist}(\omega,\partial \omega^{\ast})>0$ be simply connected domains, and let \cref{ass:simply-connected} be satisfied. Then, there exist $n_0: = C_1|\omega^{\ast}|/(\delta^{\ast})^{d}$ and $b:=\frac{1}{2}\big(C_{1}|\omega^{\ast}|/(\delta^{\ast})^{d}\big)^{-1/(d+1)}$, such that if $n>n_0$,
\begin{equation}\label{eq:exponential_decay_nwidth}
  d_{n}(P)\leq  \Vert \Xi_{h}\Vert \max\{\delta^{\ast}, 1/\delta^{\ast} \}\, e^{-bn^{1/(d+1)}},  
\end{equation}
where the constant $C_1$ depends only on $d$, the shape of $\Omega$, $\kappa_{\rm min}$, $\kappa_{\rm max}$, $\nu_{\rm min}$, $\nu_{\rm max}$, the polynomial degree $r$, and the shape-regularity of the mesh $\mathcal{T}_h$. Moreover, if \cref{ass:POU} holds, then
\begin{equation}
\Vert \Xi_{h}\Vert\leq C_2(1+\rho^{-1}),    
\end{equation}
where $C_2$ depends only on the coefficient bounds and the shape-regularity of $\mathcal{T}_h$.
\end{theorem}


The proof of \cref{thm:exponential_convergence} relies on two important properties of the discrete harmonic space ${\bm W}_{h}(D)$ (defined following \cref{eq:harmonic-space} for a general domain $D$): a Caccioppoli-type inequality and a weak approximation property, which are stated in the following two lemmas. We note that the assumption on the topology of the domains is only relevant to the weak approximation property.

\begin{lemma}[Discrete Caccioppoli inequality]\label{lem:caccioppoli_inequality}
Let $D\subset D^{\ast}$ be subdomains of $\Omega$ with $\delta:={\rm dist}\big(D, \, \partial D^{\ast}\setminus\partial \Omega\big)>0$, and let $h_{K}\leq \delta/3$ for each $K\in \mathcal{T}_{h}$ with $K\cap (D^{\ast}\setminus D)\neq \emptyset$. Then, for any ${\bm v}_{h}\in {\bm W}_{h}(D^{\ast})$,
\begin{equation}
 \Vert {\bm v}_h\Vert_{a,D}\leq C_{\rm cac}\delta^{-1}\Vert {\bm v}_h\Vert_{{\bm L}^{2}(D^{\ast}\setminus D)}, 
\end{equation}
where the constant $C_{\rm cac}$ depends on $d$, the polynomial degree $r$, the coefficient bounds, and the shape-regularity of the mesh $\mathcal{T}_{h}$. 
\end{lemma}

\begin{lemma}[Weak approximation property]\label{lem:weak_approximation}
Let $D\subset D^{\ast}\subset\subset \Omega$ be simply connected domains with $\delta:={\rm dist} ({D},\partial D^{\ast}) > 0$, and let \cref{ass:simply-connected} be satisfied. Suppose that $h\leq \delta/4$. Then there exists a family of subspaces ${\bm V}_{H,\epsilon}(D)$ of ${\bm W}_{h}(D)$, parameterized by $H\in (0,\delta)$ and $\epsilon\in (0,1)$, such that ${\rm dim} \, {\bm V}_{H,\epsilon}(D)\leq C_{{\rm dim}}\big(|D^{\ast}|/H^{d} + (|D^{\ast}|/\delta^{d})|\log \epsilon|^{d+1}\big)$, and that for any ${\bm u}_{h}\in {\bm W}_{h}(D^{\ast})$,
\begin{align}\label{eq:weak_approx_estimate}
\begin{split}
&\inf_{{\bm z}_{h}\in {\bm V}_{H,\epsilon}(D)} \Vert {\bm u}_{h} - {\bm z}_{h}\Vert_{{\bm L}^{2}(D)}\\
&\leq C_{\rm app} \big(H(\Vert \nabla \times {\bm u}_{h}\Vert_{{\bm L}^{2}(D^{\ast})} + \delta^{-1}\Vert {\bm u}_{h}\Vert_{{\bm L}^{2}(D^{\ast})}) + \epsilon \Vert {\bm u}_{h}\Vert_{{\bm L}^{2}(D^{\ast})}\big),    
\end{split}    
\end{align}
where the constants $C_{\rm dim}$ and $C_{\rm app}$ depend only on $d$, $\kappa_{\rm max}/\kappa_{\rm min}$, the shape-regularity of the mesh $\mathcal{T}_{h}$, and the shape of the domain $\Omega$ ($C_{\rm app}$ only). 
\end{lemma}

We prove \cref{lem:caccioppoli_inequality,lem:weak_approximation} in \cref{subsec:proof of Cac-inequality,subsec:proof_weak_approx}, respectively. The proof of \cref{thm:exponential_convergence} is given in \cref{subsec:proof_main_theorem} based on these two lemmas.

\subsection{Proof of the discrete Caccioppoli inequality}\label{subsec:proof of Cac-inequality}
In this subsection, we prove \cref{lem:caccioppoli_inequality}. Whereas the discrete Caccioppoli inequality has been proved for the lowest-order N\'{e}d\'{e}lec element (see, e.g., \cite[Lemma 4.1]{faustmann2022h}), we consider arbitrary high-order elements here, which is particularly relevant to the approximation of Maxwell's equations with large wavenumbers. A key tool for our proof are the following super-approximation estimates for N\'{e}d\'{e}lec elements.

\begin{lemma}\label{lem:super_approximation}
Let $\eta\in C^{r+1}(\Omega)$ satisfy $|\eta|_{W^{j,\infty}(\Omega)}\leq C\delta^{-j}$ for each $0\leq j\leq r+1$. Then for each ${\bm v}_{h}\in {\bm Q}_{h}(\Omega)$ and each element $K$ with $h_{K} :={\rm diam}(K)\leq \delta$, 
\begin{align}
{\displaystyle \Vert \nabla\times(\eta^{2}{\bm v}_{h}- \Pi_{h}(\eta^{2}{\bm v}_{h}))\Vert_{{\bm L}^{2}(K)}\leq C\big(\frac{h_{K}}{\delta}\Vert \nabla \times (\eta {\bm v}_{h})\Vert_{{\bm L}^{2}(K)}+\frac{h_{K}}{\delta^{2}}\Vert {\bm v}_{h}\Vert_{{\bm L}^{2}(K)}\big),\label{eq:super_approx_1}}\\
{\displaystyle \Vert \eta^{2}{\bm v}_{h}- \Pi_{h}(\eta^{2}{\bm v}_{h})\Vert_{{\bm L}^{2}(K)}\leq C\big(\frac{h^{2}_{K}}{\delta}\Vert \nabla \times (\eta {\bm v}_{h})\Vert_{{\bm L}^{2}(K)}+\frac{h^{2}_{K}}{\delta^{2}}\Vert {\bm v}_{h}\Vert_{{\bm L}^{2}(K)}\big),\label{eq:super_approx_2}}
\end{align}
where $\Pi_{h}$ is the N\'{e}d\'{e}lec interpolation operator.    
\end{lemma}

We postpone the proof of \cref{lem:super_approximation} to the end of this subsection. 

\begin{proof}[Proof of \cref{lem:caccioppoli_inequality}]
Let $D_{e}$ and $D^{\ast}_{e}$ be the union of elements in $\mathcal{T}_{h}$ that intersect $D$ and that are contained in $D^{\ast}$, respectively. By the assumption $\delta:={\rm dist}\,\big(D, \, \partial D^{\ast}\setminus\partial \Omega\big)\geq 3\max_{K\cap (D^{\ast}\setminus D)\neq \emptyset}h_{K}$, we have $D\subset D_{e}\subset D^{\ast}_{e}\subset D^{\ast}$ and ${\rm dist}\,\big(D_{e}, \, \partial D^{\ast}_{e}\setminus\partial \Omega\big) \geq \frac{1}{3}\delta$. Let $\eta\in C^{\infty}(\overline{{D}_{e}^{\ast}})$ be a cut-off function such that $\eta \equiv 1$ in $D_{e}$, $\eta = 0$ on $\partial {D}_{e}^{\ast}\setminus\partial \Omega$, and $|\eta|_{W^{j,\infty}(D_{e}^{\ast})}\leq C\delta^{-j}$ for $0\leq j\leq r+1$. A direct calculation yields that for any ${\bm v}_{h}\in {\bm Q}_{h}({D}_{e}^{\ast})$, 
\begin{equation}\label{eq:identity}
\begin{split}
\big\Vert \eta{\bm v}_h\big\Vert^{2}_{a, {D}_{e}^{\ast}} &= \big\Vert \nu^{1/2}\nabla \eta\times {\bm v}_h\big\Vert^{2}_{{\bm L}^{2}({D}_{e}^{\ast})} + a_{{D}_{e}^{\ast}}({\bm v}_h,\eta^{2}{\bm v}_h) \\
&\leq C\nu_{\rm max}\delta^{-1} \Vert {\bm v}_h\big\Vert^{2}_{{\bm L}^{2}({D}_{e}^{\ast}\setminus D_e)} +  a_{{D}_{e}^{\ast}}({\bm v}_h,\eta^{2}{\bm v}_h).    
\end{split}
\end{equation}
It remains to estimate the term $a_{{D}_{e}^{\ast}}({\bm v}_h,\eta^{2}{\bm v}_h)$. Let $\Pi_{h}$ be the N\'{e}d\'{e}lec interpolation operator. Note that for ${\bm v}_{h}\in {\bm W}_{h}({D}_{e}^{\ast})$, $\Pi_{h}(\eta^{2}{\bm v}_h) \in {\bm Q}_{h,0}({D}_{e}^{\ast})$ and thus $a_{{D}_{e}^{\ast}}({\bm v}_h,\Pi_{h}(\eta^{2}{\bm v}_h)) = 0$. Using this equation and noting that $\eta \equiv 1$ in $D_{e}$, we find
\begin{equation}\label{eq:identity_2}
\begin{split}
    a_{{D}_{e}^{\ast}}&({\bm v}_h,\eta^{2}{\bm v}_h) = \big(\nu \nabla\times {\bm v}_h,  \nabla\times(\eta^{2}{\bm v}_h- \Pi_{h}(\eta^{2}{\bm v}_h))\big)_{{\bm L}^{2}({D}_{e}^{\ast}\setminus D_{e})} \\
    &+    \big(\kappa {\bm v}_h,  \eta^{2}{\bm v}_h- \Pi_{h}(\eta^{2}{\bm v}_h)\big)_{{\bm L}^{2}({D}_{e}^{\ast}\setminus D_{e})} := I_1 + I_2.
\end{split}
\end{equation}
The estimate of the two terms $I_1$ and $I_2$ above relies on the super-approximation results. Using \cref{eq:super_approx_1} and an inverse inequality (local to each element), we have 
\begin{equation*}
\begin{split}
I_1\leq &  \sum_{K\subset D_{e}^{\ast}\setminus D_{e}}C\nu_{\rm max} h_{K}\Vert \nabla\times {\bm v}_h\Vert_{{\bm L}^{2}(K)} \big(\delta^{-1} \Vert\nabla \times (\eta {\bm v}_h)\Vert_{{\bm L}^{2}(K)} + \delta^{-2} \Vert {\bm v}_h\Vert_{{\bm L}^{2}(K)}\big) \\
\leq & C\nu_{\rm max}^{2}/(\nu_{\rm min}\delta^{2})\sum_{K\subset D_{e}^{\ast}\setminus D_{e}}  \Vert {\bm v}_h\Vert^{2}_{{\bm L}^{2}(K)} + \frac{\nu_{\rm min}}{4}\sum_{K \subset D_{e}^{\ast}}  \Vert\nabla \times (\eta {\bm v}_h)\Vert_{{\bm L}^{2}(K)}\\
 \leq & C\nu_{\rm max}^{2}/(\nu_{\rm min}\delta^{2})\Vert {\bm v}_h\Vert^{2}_{{\bm L}^{2}(D_{e}^{\ast}\setminus D_{e})} + \frac{1}{4}\Vert \eta {\bm v}_{h}\Vert^{2}_{a, D_e^{\ast}}.
\end{split}
\end{equation*}
Similarly, we can prove that $I_{2}\leq C\kappa_{\rm max} \Vert {\bm v}_h\Vert^{2}_{{\bm L}^{2}(D_{e}^{\ast}\setminus D_{e})}$. Combining \cref{eq:identity}, \cref{eq:identity_2}, and the estimates for $I_1$ and $I_2$ gives the desired estimate.
\end{proof}

It remains to prove \cref{lem:super_approximation} where \cref{lem:equiv_norm} plays a key role.
\begin{proof}[Proof of \cref{lem:super_approximation}]
We only prove \cref{eq:super_approx_1}. The other estimate can be proved using similar arguments. For ease of notation, we drop the dependence of the norms and the element-size on $K$. Using error estimates \cref{eq:error_nedelec_interpolation_1} for the interpolation operator $\Pi_{h}$, we obtain
\begin{equation}\label{eq:sum_two_terms}
\begin{split}
    \Vert \nabla\times (\eta^{2}{\bm v}_{h}- \Pi_{h}(\eta^{2}{\bm v}_{h}))\Vert_{{\bm L}^{2}}\leq Ch^{r+d/2} |\nabla \times (\eta^{2}{\bm v}_{h})|_{{\bm W}^{r,\infty}}\\
    \leq Ch^{r+d/2}\sum_{|{\bm \alpha} |=r} \big(\Vert \partial^{\bm \alpha} (\nabla(\eta^{2})\times{\bm v}_h) \Vert_{{\bm L}^{\infty}} + \Vert\partial^{\bm \alpha}(\eta^{2}\nabla\times{\bm v}_h) \Vert_{{\bm L}^{\infty}}\big).
\end{split}
\end{equation}
It is easy to see that
\begin{equation}\label{eq:first_term_1}
\begin{split}
\sum_{|{\bm \alpha} |=r}\Vert \partial^{\bm \alpha} (\nabla(\eta^{2})\times{\bm v}_h) \Vert_{{\bm L}^{\infty}}& \\
\leq \sum_{\substack{|{\bm \alpha} | + |{\bm \beta}|=r,\\  |{\bm \beta}|\leq r-1}}\Vert \partial^{\bm \alpha} (2\eta \nabla\eta) \Vert_{{\bm L}^{\infty}} \Vert \partial^{\bm \beta} {\bm v}_h \Vert_{{\bm L}^{\infty}}+&\sum_{|{\bm \beta}|= r}\Vert \nabla (\eta^{2}) \times \partial^{\bm \beta} {\bm v}_h \Vert_{{\bm L}^{\infty}}. 
\end{split}
\end{equation}
Using the assumptions on $\eta$ and an inverse inequality, we can prove
\begin{equation}\label{eq:first_term_2}
     Ch^{r+d/2} \sum_{\substack{|{\bm \alpha} | + |{\bm \beta}|=r,\\  |{\bm \beta}|\leq r-1}}\Vert \partial^{\bm \alpha} (2\eta \nabla\eta) \Vert_{{\bm L}^{\infty}} \Vert \partial^{\bm \beta} {\bm v}_h \Vert_{{\bm L}^{\infty}} \leq Ch\delta^{-2} \Vert{\bm v}_h\Vert_{{\bm L}^{2}}.
\end{equation}
To estimate the second term in \cref{eq:first_term_1}, we define $\widehat{\eta}:= |K|^{-1}\int_{K}\eta \,d{\bm x}$. Using the following error estimate: $\Vert \eta - \widehat{\eta}\Vert_{L^{\infty}}\leq Ch|\eta|_{W^{1,\infty}}\leq Ch\delta^{-1}$, we have
\begin{equation}\label{eq:first_term_3}
\begin{split}
 \Vert \nabla (\eta^{2}) \times \partial^{\bm \beta} {\bm v}_h \Vert_{{\bm L}^{\infty}}\leq C\delta^{-1}\Vert \eta \partial^{\bm \beta} {\bm v}_h \Vert_{{\bm L}^{\infty}}\\
 \leq C\delta^{-1}\big(\Vert (\eta-\widehat{\eta}) \partial^{\bm \beta} {\bm v}_h \Vert_{{\bm L}^{\infty}} + \Vert \widehat{\eta} \partial^{\bm \beta} {\bm v}_h \Vert_{{\bm L}^{\infty}}\big)\\
 \leq Ch\delta^{-2} \Vert \partial^{\bm \beta} {\bm v}_h \Vert_{{\bm L}^{\infty}} + C\delta^{-1} \Vert \widehat{\eta} \partial^{\bm \beta} {\bm v}_h \Vert_{{\bm L}^{\infty}}.
\end{split}
\end{equation}
It follows from \cref{eq:first_term_3}, an inverse inequality, and \cref{lem:equiv_norm} that
\begin{equation}\label{eq:first_term_4}
\begin{split}
  h^{r+d/2}  \sum_{|{\bm \beta}|= r}  &\Vert \nabla (\eta^{2}) \times \partial^{\bm \beta} {\bm v}_h \Vert_{{\bm L}^{\infty}} \leq Ch\delta^{-2} \Vert{\bm v}_h\Vert_{{\bm L}^{2}}
  + C h \delta^{-1} \Vert \widehat{\eta}\,\nabla\times {\bm v}_h \Vert_{{\bm L}^{2}}\\
 \leq  Ch\delta^{-2} &\Vert{\bm v}_h\Vert_{{\bm L}^{2}} + Ch\delta^{-1} \big(  \Vert \nabla\times  ((\widehat{\eta}-\eta){\bm v}_h) \Vert_{{\bm L}^{2}} + \Vert \nabla\times  (\eta{\bm v}_h) \Vert_{{\bm L}^{2}}\big)\\
    \leq Ch\delta^{-2} &\Vert{\bm v}_h\Vert_{{\bm L}^{2}} + Ch\delta^{-1} \Vert \nabla\times  (\eta{\bm v}_h) \Vert_{{\bm L}^{2}}.
\end{split}
\end{equation}
Inserting \cref{eq:first_term_2,eq:first_term_4} into \cref{eq:first_term_1} gives
\begin{equation}\label{eq:first_term_5}
h^{r+d/2}\sum_{|{\bm \alpha} |=r}\Vert \partial^{\bm \alpha} (\nabla(\eta^{2})\times{\bm v}_h) \Vert_{{\bm L}^{\infty}} \leq  Ch\delta^{-2} \Vert{\bm v}_h\Vert_{{\bm L}^{2}} + Ch\delta^{-1} \Vert \nabla\times  (\eta{\bm v}_h) \Vert_{{\bm L}^{2}}.   
\end{equation}
It remains to estimate the second term in \cref{eq:sum_two_terms}. Noting that $\partial^{\bm \beta} (\nabla\times {\bm v}_h) = {\bm 0}$ for any $|\bm \beta|=r$, we have
\begin{equation}\label{eq:second_term_1}
    \sum_{|{\bm \alpha} |=r}\Vert \partial^{\bm \alpha} (\eta^{2}\nabla\times{\bm v}_h) \Vert_{{\bm L}^{\infty}}
    = \sum_{\substack{|{\bm \alpha} | + |{\bm \beta}|=r,\\  |{\bm \beta}|\leq r-1}}\Vert \partial^{\bm \alpha} (\eta^{2}) \partial^{\bm \beta} \nabla\times {\bm v}_h \Vert_{{\bm L}^{\infty}}. 
\end{equation}
We estimate the sum on the RHS of \cref{eq:second_term_1} similarly as above to obtain
\begin{equation}\label{eq:second_term_2}
    h^{r+d/2} \sum_{|{\bm \alpha} |=r}\Vert \partial^{\bm \alpha} (\eta^{2}\nabla\times{\bm v}_h) \Vert_{{\bm L}^{\infty}} \leq Ch\delta^{-2} \Vert{\bm v}_h\Vert_{{\bm L}^{2}} + Ch\delta^{-1} \Vert \nabla\times  (\eta{\bm v}_h) \Vert_{{\bm L}^{2}}.
\end{equation}
Estimate \cref{eq:super_approx_1} then follows by inserting \cref{eq:first_term_5,eq:second_term_2} into \cref{eq:sum_two_terms}.    
\end{proof}

\subsection{Proof of the weak approximation property}\label{subsec:proof_weak_approx}
We first assume that $D\subset D^{\ast}$ are general subdomains of $\Omega$ with ${\rm dist}(D, \partial D^{\ast}\setminus \partial \Omega) = \delta$, and prove several auxiliary lemmas. The proof of \cref{lem:weak_approximation} will be given at the end of this subsection where we restrict ourselves to the special case that $D\subset D^{\ast}\subset\subset \Omega$ are simply connected domains. Before proceeding, we recall the local spaces defined in \cref{eq:local_spaces}. 

Let $D^{+}$ with $D\subset D^{+}\subset D^{\ast}$ be a domain such that ${\rm dist}({D}, \partial D^{+}\setminus \partial \Omega) = {\rm dist}({D}^{+}, \partial D^{\ast}\setminus \partial \Omega) = \delta /2$, and let $\widetilde{D}^{+} \in \mathscr{E}_h(D^{+})$ (see \cref{eq:set_of_domains} for the definition of $\mathscr{E}_h(D^{+})$) be a union of elements. By the assumption $h\leq \delta/4$, we see that $\widetilde{D}^{+}\subset D^{\ast}$ with ${\rm dist}(\widetilde{D}^{+}, \partial D^{\ast}\setminus \partial \Omega)\geq \delta/4$. Our proof relies on the regular decomposition of the vector field ${\bm u}_{h} \in {\bm W}_{h}(D^{\ast})$. To avoid dependence of the stability of the decomposition on the domain $D^{\ast}$, we shall use the technique of a cut-off function. Let $\eta\in W^{1,\infty}(\Omega)$ be a cut-off function such that 
\begin{equation}\label{eq:cut-off-estimate}
 \eta \equiv 1 \;\;\text{on} \;\; \widetilde{D}^{+},\quad {\rm supp}(\eta) \subset \overline{D^{\ast}}, \quad \Vert \nabla \eta\Vert_{{\bm L}^{\infty}(\Omega)}\lesssim \delta^{-1}.  
\end{equation}
It follows that ${\eta} {\bm u}_{h}\in {\bm H}_{0}({\rm curl};\Omega)$. We can apply the regular decomposition (see \cref{lem:regular_decomposition}) to ${\eta} {\bm u}_{h}$, and write it as ${\eta} {\bm u}_{h}={\bm w} + \nabla p$ on $\Omega$, where ${\bm w}\in {\bm H}^{1}_{0}(\Omega)$ and $p\in H_{0}^{1}(\Omega)$. Combining \cref{eq:cut-off-estimate} and the stability estimates \cref{eq:estimate_regular_decomp} gives
\begin{align}\label{eq:estimates_reg_decomposition}
\begin{split}
\Vert \nabla {\bm w}\Vert_{{\bm L}^{2}(\Omega)}\leq C\Vert \nabla\times (\eta {\bm u}_{h})\Vert_{{\bm L}^{2}(\Omega)} &\leq C(\Vert \nabla\times {\bm u}_{h}\Vert_{{\bm L}^{2}(D^{\ast})} + \delta^{-1}\Vert {\bm u}_{h}\Vert_{{\bm L}^{2}(D^{\ast})}),\\    
\Vert {\bm w}\Vert_{{\bm L}^{2}(\Omega)} +  \Vert \nabla p\Vert_{{\bm L}^{2}(\Omega)}&\leq C\Vert \eta {\bm u}_{h}\Vert_{{\bm L}^{2}(\Omega)}\leq C\Vert {\bm u}_{h}\Vert_{{\bm L}^{2}(D^{\ast})},
\end{split}    
\end{align}
where the constant $C$ depends only on shape of $\Omega$, but not on its diameter. 

In order to prove the weak approximation property, we use ideas from \cite{faustmann2022h} to establish a suitable decomposition for ${\bm u}_{h}$. Let ${P}_{\widetilde{D}^{+}}$ be the orthogonal projection of ${\bm L}^{2}(\widetilde{D}^{+})$ onto $\nabla S_{h,\Gamma}(\widetilde{D}^{+})$ with respect to the weighted $L^{2}$ norm $(\kappa\cdot,\cdot)_{{\bm L}^{2}(\widetilde{D}^{+})}$. Moreover, we let $p_{h},\psi_{h}\in S_{h,\Gamma}(\widetilde{D}^{+})$ such that ${P}_{\widetilde{D}^{+}} \nabla p =  \nabla p_{h}$ and ${P}_{\widetilde{D}^{+}} {\bm w} = \nabla \psi_{h}$, i.e.,
\begin{align}\label{eq:orthogonal_projection}
\begin{split}
    (\kappa \nabla p_{h},\nabla v_h)_{{\bm L}^{2}(\widetilde{D}^{+})} =  (\kappa \nabla p,\nabla v_h)_{{\bm L}^{2}(\widetilde{D}^{+})}\quad &\forall v_{h}\in S_{h,\Gamma}(\widetilde{D}^{+}),\\    
    (\kappa \nabla \psi_{h},\nabla v_h)_{{\bm L}^{2}(\widetilde{D}^{+})} =  (\kappa {\bm w},\nabla v_h)_{{\bm L}^{2}(\widetilde{D}^{+})}\quad &\forall v_{h}\in S_{h,\Gamma}(\widetilde{D}^{+}).
\end{split}    
\end{align}    
Let ${\bm w}_{h}:={\bm u}_{h}|_{\widetilde{D}^{+}}-\nabla p_{h} \in {\bm Q}_{h,\Gamma}(\widetilde{D}^{+})$. Since $\eta \equiv 1$ on $\widetilde{D}^{+}$, we have ${\bm u}_{h} = {\bm w} + \nabla p = {\bm w}_{h} + \nabla p_{h}$ on $\widetilde{D}^{+}$. Then we can decompose ${\bm u}_{h}$ on $\widetilde{D}^{+}$ as 
\begin{equation}\label{eq:four_terms_decomposition}
    {\bm u}_{h} = {\bm w}_{h} + \nabla p_{h} = ({\bm w}_{h} - {\bm w}) + \nabla (p_{h} + \psi_{h}) +  {\bm w} - \nabla \psi_{h}.
\end{equation}
In the following, we approximate the four terms on the right-hand side of \cref{eq:four_terms_decomposition}, based on the observations: (i) ${\bm w}_{h}$ and ${\bm w}$ are close; (ii) $p_{h} + \psi_{h}$ is discrete harmonic on $\widetilde{D}^{+}$ and therefore can be approximated in $D$ with exponential accuracy; (iii) ${\bm w}\in {\bm H}^{1}$ can be approximated in ${\bm L}^2$ using usual techniques; (iv) $\nabla \psi_{h}$ is the $L^{2}$ projection of ${\bm w}$ onto $\nabla S_{h,\Gamma}(\widetilde{D}^{+})$, and hence its approximation follows from (iii). To approximate ${\bm w}_{h} - {\bm w}$, we introduce the space of discrete `harmonic forms' (see \cref{subsec:preliminaries}):
\begin{align}\label{eq:discrete-harmonic-forms}
\begin{split}
\mathcal{\bm H}_{h}(\widetilde{D}^{+}):= \big\{{\bm v}_h\in {\bm Q}_{h,\Gamma}(\widetilde{D}^{+}): \nabla\times {\bm v}_h = {\bm 0},\;\\
(\kappa {\bm v}_h, \nabla \xi_h)_{{\bm L}^{2}(\widetilde{D}^{+})} = 0 \quad \forall\,\xi_h\in S_{h,\Gamma}(\widetilde{D}^{+})\big\}.       
\end{split}
\end{align}


\begin{lemma}\label{lem:estimate_term_1}
Let ${\bm w}$, ${\bm w}_{h}$, and $\mathcal{\bm H}_{h}(\widetilde{D}^{+})$ be defined as above. Then
\begin{equation*}
\inf_{{\bm v}_h\in \mathcal{\bm H}_{h}(\widetilde{D}^{+})}\Vert {\bm w} - {\bm w}_{h}-{\bm v}_h \Vert_{{\bm L}^{2}(\widetilde{D}^{+})} \leq C\Big(\frac{\kappa_{\rm max}}{\kappa_{\rm min}}\Big)^{1/2}h\big(\Vert \nabla {\bm w}\Vert_{{\bm L}^{2}(\widetilde{D}^{+})} + \Vert \nabla \times {\bm w}\Vert_{{\bm L}^{2}(\widetilde{D}^{+})} \big),    
\end{equation*}
where $C$ depends only on the shape-regularity of the mesh. 
\end{lemma}
\begin{proof}
Let $\Pi_{h}$ and $\Sigma_h$ be the N\'{e}d\'{e}lec and Raviart--Thomas interpolation operators, respectively, and let ${\bm U}_{h,\Gamma}(\widetilde{D}^{+})$ be the Raviart-Thomas element space defined in \cref{eq:local_spaces}. Since ${\bm w}\in {\bm H}^{1}(\Omega)$ and $\nabla \times {\bm w} = \nabla \times {\bm u}_{h}$ in $\widetilde{D}^{+}$, we see that the interpolant $\Pi_{h}{\bm w}$ is well defined. Moreover, using \cref{eq:commuting_property} and the fact that $\nabla\times {\bm w} = \nabla\times {\bm w}_h$ gives
\begin{equation*}
\nabla\times (\Pi_{h}{\bm w}) = \Sigma_h(\nabla\times {\bm w}) =  \Sigma_h(\nabla\times {\bm w}_h) = \nabla\times {\bm w}_h,   
\end{equation*}
i.e., $\nabla \times ({\bm w}_{h} - \Pi_{h}{\bm w}) = 0$. Hence, we have the decomposition ${\bm w}_{h} - \Pi_{h}{\bm w} = \nabla \varphi_{h} + {\bm z}_h$, with $\varphi_{h}\in S_{h,\Gamma}(\widetilde{D}^{+})$ and ${\bm z}_h\in \mathcal{\bm H}_{h}(\widetilde{D}^{+})$. Using this decomposition and \cref{eq:orthogonal_projection} yields
\begin{align*}
\begin{split}
 \big(\kappa ({\bm w} &- {\bm w}_{h} - {\bm z}_h),\, {\bm w}_{h} - \Pi_{h}{\bm w} - {\bm z}_h\big)_{{\bm L}^{2}(\widetilde{D}^{+})} =   \big(\kappa ({\bm w} - {\bm w}_{h}-{\bm z}_h),\, \nabla \varphi_h\big)_{{\bm L}^{2}(\widetilde{D}^{+})}   \\
& = \big(\kappa (\nabla p_h - \nabla p),\, \nabla \varphi_h\big)_{{\bm L}^{2}(\widetilde{D}^{+})}  - \big(\kappa {\bm z}_h,\, \nabla \varphi_h\big)_{{\bm L}^{2}(\widetilde{D}^{+})} = 0. 
\end{split}    
\end{align*}
It follows that
\begin{align*}
\begin{split}
\big(\kappa ({\bm w}_{h} &- {\bm w} - {\bm z}_h), \,{\bm w}_{h} - {\bm w}-{\bm z}_h\big)_{{\bm L}^{2}(\widetilde{D}^{+})} =\big(\kappa ({\bm w} - {\bm w}_{h}-{\bm z}_h),\, \Pi_{h}{\bm w} - {\bm w}\big)_{{\bm L}^{2}(\widetilde{D}^{+})}\\
&\leq \Vert \kappa^{1/2}({\bm w}_{h} - {\bm w}-{\bm z}_h)\Vert_{{\bm L}^{2}(\widetilde{D}^{+})} \Vert  \kappa^{1/2}(\Pi_{h}{\bm w} - {\bm w})\Vert_{{\bm L}^{2}(\widetilde{D}^{+})},
\end{split}
\end{align*}
which leads to
\begin{equation}\label{eq:mid_estimate}
 \Vert {\bm w} - {\bm w}_{h}- {\bm z}_h\Vert_{{\bm L}^{2}(\widetilde{D}^{+})} \leq (\kappa_{\rm max}/\kappa_{\rm min})^{1/2} \Vert\Pi_{h}{\bm w} - {\bm w} \Vert_{{\bm L}^{2}(\widetilde{D}^{+})}.   
\end{equation}
Noting that $\nabla \times {\bm w} = \nabla \times {\bm u}_{h}\in {\bm U}_{h,\Gamma}(\widetilde{D}^{+})$, we use the error estimate \cref{eq:error_nedelec_interpolation_2} to derive
\begin{equation}\label{eq:L2error_estimate}
\Vert\Pi_{h}{\bm w} - {\bm w} \Vert_{{\bm L}^{2}(\widetilde{D}^{+})} \leq Ch(\Vert \nabla {\bm w}\Vert_{{\bm L}^{2}(\widetilde{D}^{+})} + \Vert \nabla \times {\bm w}\Vert_{{\bm L}^{2}(\widetilde{D}^{+})}).    
\end{equation}
Combining \cref{eq:mid_estimate,eq:L2error_estimate} concludes the proof of the lemma.
\end{proof}
\begin{remark}
If the domain $\widetilde{D}^{+}$ is simply connected and $\partial \widetilde{D}^{+}\cap \Gamma$ is connected, then $\mathcal{\bm H}_{h}(\widetilde{D}^{+})=\{ 0\}$.      
\end{remark}

Next we approximate the term $\nabla (p_{h} + \psi_{h})$. Using \cref{eq:orthogonal_projection} and recalling the regular decomposition, we see that 
\begin{align}
\begin{split}
(\kappa (\nabla p_{h}&+\psi_{h}), \,\nabla v_h)_{{\bm L}^{2}(\widetilde{D}^{+})}  = (\kappa({\bm w} + \nabla p), \,\nabla v_h)_{{\bm L}^{2}(\widetilde{D}^{+})} \\
&= (\kappa{\bm u}_{h}, \,\nabla v_h)_{{\bm L}^{2}(\widetilde{D}^{+})} = 0 \quad \forall v_{h}\in S_{h,0}(\widetilde{D}^{+}),    
\end{split}
\end{align}
i.e., $p_{h} + \psi_{h}$ lies in the following scalar discrete harmonic space 
\begin{equation}\label{eq:scalar-harmonic-space}
W_{h}(\widetilde{D}^{+}):=\big\{u_h\in S_{h,\Gamma}(\widetilde{D}^{+}): (\kappa \nabla u_h,\,\nabla v_h)_{{\bm L}^{2}(\widetilde{D}^{+})} = 0 \quad \forall v_h\in S_{h,0}(\widetilde{D}^{+})\big\}.    
\end{equation}
We can approximate $\nabla(p_{h} + \psi_{h})$ using previous results for scalar elliptic equations. 
\begin{lemma}\label{lem:estimate_term_2}
Let $p_{h}$ and $\psi_{h}$ be defined by \cref{eq:orthogonal_projection}. Then for $\epsilon \in (0,1)$, there exists a space ${\bm V}_{\epsilon}(D)\subset {\bm L}^{2}(D)$ with ${\rm dim}\,{\bm V}_{\epsilon}(D)\leq C_1(|\widetilde{D}^{+}|/\delta^{d}) |\log \epsilon|^{d+1}$ such that  
\begin{equation}\label{eq:harmonic_approx}
\inf_{{\bm v}\in {\bm V}_{\epsilon}(D)}\Vert \nabla(p_{h} + \psi_{h}) - {\bm v}\Vert_{{\bm L}^{2}(D)} \leq  C_2\epsilon (\Vert {\bm w}\Vert_{{\bm L}^{2}(\widetilde{D}^{+})} +  \Vert \nabla p\Vert_{{\bm L}^{2}(\widetilde{D}^{+})}),    
\end{equation}
where $C_1$ and $C_2$ depend only on $d$, ${\kappa_{\rm max}}/{\kappa_{\rm min}}$, the polynomial degree $r$, and the shape-regularity of the mesh $\mathcal{T}_h$.
\end{lemma}

\begin{proof}
We use established exponential decay estimates for a Kolmogorov $n$-width for Poisson-type problems. Noting that ${\rm dist}({D}, \,\partial \widetilde{D}^{+}\setminus \partial \Omega)\geq \delta/2$, by \cite[Theorem 3.8]{ma2025unified}, we know that the following Kolmogorov $n$-width 
\begin{equation*}
d_{n}(D, \widetilde{D}^{+}): = \inf_{Q(n)\subset W_{h}(D)}\sup_{u\in W_{h}(\widetilde{D}^{+})} \inf_{v_h\in Q(n)}\frac {\Vert \kappa^{1/2}\nabla (u_h-v_h)\Vert_{{\bm L}^{2}(D)}}{\Vert \kappa^{1/2}\nabla u_h\Vert_{{\bm L}^{2}(\widetilde{D}^{+})}}    
\end{equation*}
satisfies that $d_{n}(D, \widetilde{D}^{+}) \leq e^{-cn^{1/(d+1)}}$, where $c = C_0(\delta^{d}/|\widetilde{D}^{+}|)^{1/(d+1)}$, with $C_0$ depending only on $d$, ${\kappa_{\rm max}}/{\kappa_{\rm min}}$, the polynomial degree $r$, and the shape-regularity of $\mathcal{T}_h$. Let $n = c^{-(d+1)}|\log \epsilon|^{d+1}$, and ${\bm V}_{\epsilon}(D) = Q^{\rm opt}(n)$, where $Q^{\rm opt}(n)$ is the optimal approximation space associated with $d_{n}(D, \widetilde{D}^{+})$. Since $p_{h} + \psi_{h}\in W_{h}(\widetilde{D}^{+})$, we have
\begin{equation}\label{eq:epsilon_estimate}
\inf_{{\bm v}\in {\bm V}_{\epsilon}(D)}\big\Vert \kappa^{1/2}\big(\nabla(p_{h} + \psi_{h}) - {\bm v}\big)\big\Vert_{{\bm L}^{2}(D)} \leq \epsilon \Vert \kappa^{1/2} \nabla(p_{h} + \psi_{h}) \Vert_{{\bm L}^{2}(\widetilde{D}^{+})}.
\end{equation}
Moreover, it follows from \cref{eq:orthogonal_projection} that $\Vert \kappa^{1/2}\nabla (p_{h} + \psi_{h}) \Vert_{{\bm L}^{2}(\widetilde{D}^{+})}\leq \kappa^{1/2}_{\rm max} \big( \Vert \nabla p\Vert_{{\bm L}^{2}(\widetilde{D}^{+})} +\Vert {\bm w}\Vert_{{\bm L}^{2}(\widetilde{D}^{+})}\big)$. Inserting this estimate into \cref{eq:epsilon_estimate} gives \cref{eq:harmonic_approx}.
\end{proof}
We proceed to approximate ${\bm w}$ and $\nabla \psi_{h}$. 
\begin{lemma}\label{lem:estimate_term_3_4}
Let ${\bm w}$ and $\phi_h$ be defined as above. Then for $0<H<\delta$, there exists a space ${\bm V}_{H}(\widetilde{D}^{+})\subset {\bm L}^{2}(\widetilde{D}^{+})$ with ${\rm dim}\,{\bm V}_{H}(\widetilde{D}^{+}) \leq C_1|D^{\ast}|/H^{d}$ such that 
\begin{align}
   \inf_{{\bm v}\in {\bm V}_{H}(\widetilde{D}^{+})}\Vert {\bm w} - {\bm v}\Vert_{{\bm L}^{2}(\widetilde{D}^{+})} &\leq C_2 H\Vert \nabla {\bm w}\Vert_{{\bm L}^{2}(D^{\ast})}, \label{eq:approximation_w}   \\
\inf_{{\bm v}\in {\bm U}_{H}(\widetilde{D}^{+})}\Vert \nabla \psi_{h} - {\bm v}\Vert_{{\bm L}^{2}(\widetilde{D}^{+})} &\leq C_2(\kappa_{\rm max}/\kappa_{\rm min})^{1/2}H\Vert \nabla {\bm w}\Vert_{{\bm L}^{2}(D^{\ast})}, \label{eq:approximation_psi}  
\end{align}
where ${\bm U}_{H}(\widetilde{D}^{+})=P_{\widetilde{D}^{+}} {\bm V}_{H}(\widetilde{D}^{+})$, and $C_1$, $C_2$ depend only on $d$.
\end{lemma}
\begin{proof}
Noting that ${\bm w}\in {\bm H}^{1}(\widetilde{D}^{+})$ and that ${\rm dist}(\widetilde{D}^{+}, \partial D^{\ast}\setminus \partial \Omega)\geq \delta/4$, we can approximate ${\bm w}$ following \cite[Lemma 5.5(ii)]{ma2025unified} by means of piecewise constant functions. Moreover, recalling that $\nabla \psi_{h} = P_{\widetilde{D}^{+}} {\bm w}$, we have $\Vert \kappa^{1/2}(\nabla \psi_{h} - P_{\widetilde{D}^{+}}{\bm v})\Vert_{{\bm L}^{2}(\widetilde{D}^{+})} \leq \Vert  \kappa^{1/2}( {\bm w} - {\bm v})\Vert_{{\bm L}^{2}(\widetilde{D}^{+})}$. Combining this inequality with \cref{eq:approximation_w} yields \cref{eq:approximation_psi}.   
\end{proof}

Now we are ready to prove \cref{lem:weak_approximation}.

\begin{proof}[\textit{Proof of \cref{lem:weak_approximation}}]
By assumption, $D\subset D^{\ast}\subset\subset\omega^{\ast}$ are simply connected domains. Therefore, we can let $D^{+}$ and thus $\widetilde{D}^{+}$ be also simply connected domains (see \cref{ass:simply-connected}). In this case, $\mathcal{\bm H}_{h}(\widetilde{D}^{+}) = \{0\}$. Let ${\bm V}_{\epsilon}(D)$, ${\bm V}_{H}(\widetilde{D}^{+})$, and ${\bm U}_{H}(\widetilde{D}^{+})$ be as in \cref{lem:estimate_term_2,lem:estimate_term_3_4}, and let $\widehat{\bm V}_{H,\epsilon}(D)={\bm V}_{\epsilon}(D) + {\bm V}_{H}(\widetilde{D}^{+})|_{D}+{\bm U}_{H}(\widetilde{D}^{+})|_{D}$ with $\epsilon\in (0,1)$ and $H\in (0,\delta)$. Then ${\rm dim}\,\widehat{\bm V}_{H,\epsilon}(D)\leq C_1 \big(|D^{\ast}|/H^{d} + (|D^{\ast}|/\delta^{d}) |\log \epsilon|^{d+1}\big)$. Using \cref{eq:four_terms_decomposition}, \cref{eq:estimates_reg_decomposition}, and \cref{lem:estimate_term_1,lem:estimate_term_2,lem:estimate_term_3_4,} leads us to
\begin{align}\label{eq:aux_final_approx}
\begin{split}
&\inf_{{\bm v}\in \widehat{\bm V}_{H,\epsilon}(D)}\Vert {\bm u}_h- {\bm v}\Vert_{{\bm L}^{2}(D)}\\
\leq C(h+H)&\big(\Vert \nabla {\bm w}\Vert_{{\bm L}^{2}(D^{\ast})} + \Vert \nabla \times {\bm w}\Vert_{{\bm L}^{2}(D^{\ast})}\big) + C\epsilon (\Vert {\bm w}\Vert_{{\bm L}^{2}(D^{\ast})} +  \Vert \nabla p\Vert_{{\bm L}^{2}(D^{\ast})})  \\
\leq C(h+H)&(\Vert \nabla\times {\bm u}_{h}\Vert_{{\bm L}^{2}(D^\ast)} + \delta^{-1}\Vert {\bm u}_{h}\Vert_{{\bm L}^{2}(D^{\ast})}) + C\epsilon \Vert {\bm u}_{h}\Vert_{{\bm L}^{2}(D^{\ast})}. 
\end{split}
\end{align}
We define the desired approximation space ${\bm V}_{H,\epsilon}(D)$ as the ${\bm L}^{2}$ orthogonal projection of $\widehat{\bm V}_{H,\epsilon}(D)$ onto ${\bm W}_{h}(D)$. Assuming $h\leq H$ and using \cref{eq:aux_final_approx} and properties of the ${\bm L}^{2}$ orthogonal projection, we obtain the estimate \cref{eq:weak_approx_estimate}. When $h>H$, we can simply choose $\widehat{\bm V}_{H,\epsilon}(D) = {\bm W}_{h}(D)$ which satisfies that ${\rm dim}\,\widehat{\bm V}_{H,\epsilon}(D)\leq C_{d}|\widetilde{D}^{+}|/h^{d}\leq C_d|D^{\ast}|/H^{d}$. In this case, the estimate holds trivially since the infimum on the left is zero.
\end{proof}  


\subsection{Proof of \cref{thm:exponential_convergence}}\label{subsec:proof_main_theorem}
We prove \cref{thm:exponential_convergence} by first combining the Caccioppoli inequality and the weak approximation property to obtain a one-step approximation in the energy norm, and then iterating the obtained result over a family of nested subdomains between $\omega$ and $\omega^{\ast}$. Although similar proofs have been well presented in previous works \cite{babuska2011optimal,ma2021novel,ma2025unified}, a special treatment is needed here due to the presence of the term $\delta^{-1}$ in the weak approximation estimate \cref{eq:weak_approx_estimate}. As mentioned above, we first combine \cref{lem:caccioppoli_inequality,lem:weak_approximation} to obtain an auxiliary approximation result. For a domain $D\subset \Omega$ and $\delta>0$, we introduce the following weighted energy norm:
\begin{equation}\label{eq:weighted_norm}
  \Vert {\bm u}\Vert_{a, D, \delta}:= \big(\Vert \nu^{1/2}\nabla\times {\bm u}\Vert^{2}_{{\bm L}^{2}(D)} + \delta^{-2}\Vert \kappa^{1/2}{\bm u}\Vert^{2}_{{\bm L}^{2}(D)}\big)^{1/2} \quad \text{for}\;\;{\bm u}\in {\bm H}({\rm curl};D).
\end{equation}
\vspace{-2ex}
\begin{lemma}\label{lem:aux_approximation}
Let $D\subset D^{\ast}\subset\subset \omega^{\ast}$ with $\delta:={\rm dist} ({D},\partial D^{\ast}\setminus \partial \Omega) > 0$ be simply connected domains, and let $H\in (0,\delta/2)$, $\epsilon\in (0,1)$, and $h\in (0,\delta/8)$. Then, there exists a subspace ${\bm V}_{H,\epsilon}(D)$ of ${\bm W}_{h}(D)$ with ${\rm dim} \, {\bm V}_{H,\epsilon}(D)\leq C^{\prime}_{\rm dim}\big( |D^{\ast}|/H^{d} + (|D^{\ast}|/\delta^{d})|\log \epsilon|^{d+1}\big)$ such that for any ${\bm v}_{h}\in {\bm W}_{h}(D^{\ast})$, 
\begin{equation}\label{eq:single_step_estimate}
\inf_{{\bm z}_h\in {\bm V}_{H,\epsilon}(D)}\Vert {\bm v}_h - {\bm z}_h\Vert_{a, D, \delta} \leq C^{\prime}_{\rm app}(H/\delta + \epsilon) \Vert {\bm v}_h\Vert_{a, D^{\ast}, \delta},    
\end{equation}
where $C^{\prime}_{\rm dim} = 2^{d} C_{\rm dim}$ and $C^{\prime}_{\rm app} = \sqrt{2}C_{\rm app}\max\big\{\nu^{-1/2}_{\rm min},2\kappa^{-1/2}_{\rm min}\big\} (2C_{\rm cac}+ \kappa^{1/2}_{\rm max})$, with $C_{\rm dim}$, $C_{\rm app}$, and $C_{\rm cac}$ from \cref{lem:weak_approximation,lem:caccioppoli_inequality}.
\end{lemma}
\begin{proof}
Let $D^{+}$ be a simply connected domain between $D$ and $D^{\ast}$ such that ${\rm dist}(D, \partial D^{+}\setminus \partial \Omega) = {\rm dist}(D^{+}, \partial D^{\ast} \setminus \partial \Omega) = \delta/2$. Applying \cref{lem:weak_approximation} to $D^{+}\subset D^{\ast}$ and using the weighted norm \cref{eq:weighted_norm}, we know that there exists a subspace $\widehat{\bm V}_{H,\epsilon}(D^{+})$ of ${\bm W}_{h}(D^{+})$ with ${\rm dim} \, \widehat{\bm V}_{H,\epsilon}(D^{+}) \leq C_{\rm dim}\big(|D^{\ast}|/H^{d} + (2^{d}|D^{\ast}|/\delta^{d})|\log \epsilon|^{d+1}\big)$ such that 
\begin{align}\label{eq:aux_1}
\begin{split}
\inf_{{\bm z}_{h}\in \widehat{\bm V}_{H,\epsilon}(D^{+})} \Vert {\bm v}_{h} - {\bm z}_{h}&\Vert_{{\bm L}^{2}(D^{+})}\leq  C_{\rm app}\big(H(\Vert \nabla \times {\bm v}_{h}\Vert_{{\bm L}^{2}(D^{\ast})} + 2\delta^{-1}\Vert {\bm v}_{h}\Vert_{{\bm L}^{2}(D^{\ast})})\\
+ \epsilon \Vert {\bm v}_{h}\Vert_{{\bm L}^{2}(D^{\ast})} \big)&\leq \sqrt{2}C_{\rm app}\max\big\{\nu^{-1/2}_{\rm min},2\kappa^{-1/2}_{\rm min}\big\}(H+\epsilon \delta)  \Vert {\bm v}_h\Vert_{a, D^{\ast}, \delta}.  
\end{split}    
\end{align}
Next we apply the Caccioppoli inequality in \cref{lem:caccioppoli_inequality} to $D\subset D^{+}$ and find
\begin{align}\label{eq:aux_2}
\begin{split}
\Vert {\bm v}_{h}\Vert_{a, D, \delta} &\leq C_{\rm cac}(\delta/2)^{-1}\Vert {\bm v}_h\Vert_{{\bm L}^{2}(D^{+}\setminus D)} + \kappa^{1/2}_{\rm max}\delta^{-1} \Vert {\bm v}_h\Vert_{{\bm L}^{2}(D)}\\
&\leq (2C_{\rm cac}+ \kappa^{1/2}_{\rm max})\delta^{-1}  \Vert {\bm v}_h\Vert_{{\bm L}^{2}(D^{+})} \quad \text{for all}\;\; {\bm v}_{h}\in {\bm W}_{h}(D^{+}).
\end{split}    
\end{align}
Combining \cref{eq:aux_1,eq:aux_2} and setting ${\bm V}_{H,\epsilon}(D):=  \widehat{\bm V}_{H,\epsilon}(D^{+})|_{D}$ yields \cref{eq:single_step_estimate}. 
\end{proof}

Next we estimate the norm of the partition of unity operators. 
\begin{lemma}\label{lem:norm_estimate_POU}
Let the partition of unity operator $\Xi_h$ be defined by \cref{eq:discrete_pou}, and let \cref{ass:POU} be satisfied. Then with $\rho>0$ defined in \cref{ass:POU},
\begin{equation}\label{eq:stability_pou}
\Vert \Xi_h({\bm v}_h)\Vert_{a,\omega}\leq C(1+\rho^{-1}) \Vert {\bm v}_h\Vert_{a,\omega} \quad \text{for all}\;\; {\bm v}_h\in {\bm Q}_{h,\Gamma}(\omega), 
\end{equation}
where $C$ depends on $\kappa_{\rm min}$, $\kappa_{\rm max}$, $\nu_{\rm min}$, $\nu_{\rm max}$, and the shape regularity of the mesh $\mathcal{T}_h$.
\end{lemma}
\begin{proof}
Using the error estimates \cref{eq:error_nedelec_interpolation_1} for the interpolation operator $\Pi_h$, we have
\begin{align}\label{eq:estimate_pou_1}
\begin{split}
    \Vert \chi {\bm v}_h -& \Pi_{h}(\chi {\bm v}_h)\Vert_{{\bm L}^{2}(K)}+ \Vert \nabla\times (\chi {\bm v}_h - \Pi_{h}(\chi {\bm v}_h))\Vert_{{\bm L}^{2}(K)} \\
    &\leq Ch \big(|\chi {\bm v}_h|_{{\bm H}^{1}(K)} + |\nabla\times (\chi {\bm v}_h)|_{{\bm H}^{1}(K)} \big).
\end{split}
\end{align}
Furthermore, it follows from \cref{ass:POU} and an inverse estimate that 
\begin{equation}\label{eq:estimate_pou_2}
h \big(|\chi {\bm v}_h|_{{\bm H}^{1}(K)} + |\nabla\times (\chi {\bm v}_h)|_{{\bm H}^{1}(K)} \big) \leq C(1+\rho^{-1}) \Vert {\bm v}_h\Vert_{{\bm H}({\rm curl};\,K)}.    
\end{equation}
Inserting \cref{eq:estimate_pou_2} into \cref{eq:estimate_pou_1}, summing the result over all elements $K\subset\omega$, and using a triangle inequality and \cref{ass:POU} again, we obtain
\begin{align}
\begin{split}
\Vert \Pi_h(\chi {\bm v}_h)\Vert_{{\bm H}({\rm curl};\omega)} &\leq \Vert \chi {\bm v}_h\Vert_{{\bm H}({\rm curl};\omega)} + \Vert \chi {\bm v}_h -  \Pi_h(\chi {\bm v}_h)\Vert_{{\bm H}({\rm curl};\omega)}   \\
&\leq C(1+\rho^{-1}) \Vert {\bm v}_h\Vert_{{\bm H}({\rm curl;\omega})},
\end{split}
\end{align}
which immediately gives \cref{eq:stability_pou}.
\end{proof}

Now we are ready to prove \cref{thm:exponential_convergence}. For $N\in \mathbb{N}$, let $\{\omega^{k} \}_{k=2}^{N}$ be a family of simply connected domains between $\omega$ and $\omega^{\ast}$ with $\omega^{N+1}:=\omega\subset \omega^{N}\subset\cdots\subset\omega^{1}:=\omega^{\ast}$ and $\delta:={\rm dist}(\omega^{k+1},\partial \omega^{k}\setminus \partial \Omega) =  \delta^{\ast}/N$. For each $1\leq k\leq N$, let ${\bm V}^{k}_{H,\epsilon}$ be the space as in \cref{lem:aux_approximation} with $D^{\ast}$ and $D$ replaced by $\omega^{k}$ and $\omega^{k+1}$, respectively, and let $\widehat{\bm V}_{\rm app} = {\bm V}_{H,\epsilon}^{1}|_{\omega} + \cdots + {\bm V}_{H,\epsilon}^{N}|_{\omega} \subset {\bm W}_{h}(\omega)$. We first assume that $h\leq \delta/8$. By repeatedly applying \cref{lem:aux_approximation}, we see that for each ${\bm v}_{h}\in {\bm W}_{h}(\omega^{\ast})$,
\begin{equation}
\inf_{{\bm z}_h\in \widetilde{\bm V}_{\rm app}}\Vert {\bm v}_h|_{\omega} - {\bm z}_h\Vert_{a, \omega, \delta} \leq \big[C^{\prime}_{\rm app}(H/\delta + \epsilon) \big]^{N} \Vert {\bm v}_h\Vert_{a, \omega^{\ast}, \delta}.    
\end{equation}
Next we choose $C^{\prime}_{\rm app}H/\delta = C^{\prime}_{\rm app}NH/\delta^{\ast} = 1/(2e)$ and $C^{\prime}_{\rm app}\epsilon = 1/(2e)$. It follows that for each $1\leq k\leq N$, ${\rm dim} \, {\bm V}_{H,\epsilon}^{k} \leq  C^{\prime\prime}_{\rm dim}N^{d}|\omega^{\ast}|/(\delta^{\ast})^{d}$ with $C^{\prime\prime}_{\rm dim}:=C^{\prime}_{\rm dim}\big((2eC^{\prime}_{\rm app})^{d} + \log^{d+1}{(2eC^{\prime}_{\rm app})}\big)$. Moreover, we have, with ${\rm dim} \, \widehat{\bm V}_{\rm app} \leq  C^{\prime\prime}_{\rm dim}N^{d+1}|\omega^{\ast}|/(\delta^{\ast})^{d}$, 
\begin{equation}\label{eq:weighted_exponential_decay}
\inf_{{\bm z}_h\in \widehat{\bm V}_{\rm app}}\Vert {\bm v}_h|_{\omega} - {\bm z}_h\Vert_{a, \omega, \delta} \leq e^{-N} \Vert {\bm v}_h\Vert_{a, \omega^{\ast},\delta}.     
\end{equation}
When $h>\delta/8$, we can simply let $\widehat{\bm V}_{\rm app} := {\bm W}_h(\omega)$. In this case, ${\rm dim} \, \widehat{\bm V}_{\rm app} \leq C_{d} |\omega|/h^{d} \leq C_{d}N^{d}|\omega|/(\delta^{\ast})^{d}$, and \cref{eq:weighted_exponential_decay} holds trivially since ${\bm v}_h|_{\omega}\in {\bm W}_h(\omega)$. Using \cref{eq:weighted_exponential_decay} and recalling the definition of the norm $\Vert \cdot\Vert_{a, \omega, \delta}$, we further have
\begin{align}
  \begin{split}
& \inf_{{\bm z}_h\in \widehat{\bm V}_{\rm app}}\Vert {\bm v}_h - {\bm z}_h\Vert_{a, \omega}\leq \max\{1,\delta\}\inf_{{\bm z}_h\in \widehat{\bm V}_{\rm app}}\Vert {\bm v}_h - {\bm z}_h\Vert_{a, \omega, \delta} \\
& \leq \max\{1,\delta\} e^{-N} \Vert {\bm v}_h\Vert_{a, \omega^{\ast}, \delta}\leq   \max\{\delta,\delta^{-1}\} e^{-N}\Vert {\bm v}_h\Vert_{a, \omega^{\ast}}, \\[1ex]
&\leq \max\{\delta^{\ast}, 1/\delta^{\ast}\} Ne^{-N}\Vert {\bm v}_h\Vert_{a, \omega^{\ast}} \leq \max\{\delta^{\ast}, 1/\delta^{\ast}\} e^{-N/2}\Vert {\bm v}_h\Vert_{a, \omega^{\ast}}.
\end{split}  
\end{align}
Let $n:={\rm dim} \widehat{\bm V}_{\rm app}\leq C^{\prime\prime}_{\rm dim}N^{d+1}|\omega^{\ast}|/(\delta^{\ast})^{d}$ and $b:= \frac{1}{2}\big(C^{\prime\prime}_{\rm dim}|\omega^{\ast}|/(\delta^{\ast})^{d}\big)^{-1/(d+1)}$. It follows that $N\geq 2b n^{1/(d+1)}$ and thus
\begin{equation}\label{eq:nonweighted_exponential_decay}
\inf_{{\bm z}_h\in \widehat{\bm V}_{\rm app}}\Vert {\bm v}_h - {\bm z}_h\Vert_{a, \omega}\leq  \max\{\delta^{\ast}, 1/\delta^{\ast}\} e^{-bn^{1/(d+1)}}\Vert {\bm v}_h\Vert_{a, \omega^{\ast}}.  
\end{equation}
Let ${\bm V}_{\rm app}:= \Xi_h(\widehat{\bm V}_{\rm app})\subset {\bm Q}_{h,0}(\omega)$. Then ${\rm dim}\,{\bm V}_{\rm app} \leq n$. Combining \cref{eq:nonweighted_exponential_decay} and \cref{lem:norm_estimate_POU} yields
\begin{align}\label{eq:aux_exponential_decay}
\begin{split}
 \inf_{{\bm z}_h\in {\bm V}_{\rm app}}&\Vert \Xi_h({\bm v}_h) - {\bm z}_h\Vert_{a, \omega} \leq \Vert\Xi_h\Vert \inf_{{\bm z}_h\in \widehat{\bm V}_{\rm app}}\Vert {\bm v}_h - {\bm z}_h\Vert_{a, \omega} \\
 \leq &C(1+\rho^{-1}) \max\{\delta^{\ast}, 1/\delta^{\ast}\} e^{-bn^{1/(d+1)}}   \Vert{\bm v}_h\Vert_{a, \omega^{\ast}}. 
\end{split}
\end{align}
To ensure $N\geq 1$, it suffices to let $n\geq n_0:=  C^{\prime\prime}_{\rm dim}|\omega^{\ast}|/(\delta^{\ast})^{d}$.  
\Cref{thm:exponential_convergence} then follows immediately from \cref{eq:nonweighted_exponential_decay}, \cref{eq:aux_exponential_decay}, and the definition of the $n$-width.

\subsection{Extension to (boundary) subdomains with nontrivial topologies}\label{subsec:nontrivial_topology} 
In this subsection, we consider the local convergence of MS-GFEM on domains that may intersect $\partial\Omega$ and have nontrivial topologies. An inspection of the proof of \cref{thm:exponential_convergence} can find that to extend the exponential decay result to this case, one only needs to incorporate the space of discrete harmonic forms $\mathcal{\bm H}_{h}(\widetilde{D}^{+})$ into the approximation space in \cref{lem:weak_approximation}. The key point is the dimension of $\mathcal{\bm H}_{h}(\widetilde{D}^{+})$. As discussed in \cref{subsec:preliminaries}, ${\rm dim}\,\mathcal{\bm H}_{h}(\widetilde{D}^{+})$ depends on the topological properties of $\widetilde{D}^{+}$ and $\partial \widetilde{D}^{+}\cap \Gamma$, but not on the underlying triangulation $\mathcal{T}_h$. Based on \cref{lem:dimension_harmonic_forms,lem:Betti-numbers}, and in view of the iteration argument used in the proof of \cref{thm:exponential_convergence}, we assume that ${\rm dim}\,\mathcal{\bm H}_{h}(\widetilde{D}^{+})$ can be bounded uniformly for a family of nested domains $\{D\}$ between $\omega$ and $\omega^{\ast}$. The following assumption is a counterpart of \cref{ass:simply-connected}.

\begin{assumption}\label{ass:dimension_harmonic_forms}
For each $N\in\mathbb{N}$, let $\{\omega_{N}^{k} \}_{k=1}^{N+1}$ be a family of domains between $\omega$ and $\omega^{\ast}$ with $\omega_{N}^{N+1}:=\omega\subset \cdots\subset\omega_{N}^{1}:=\omega^{\ast}$ and ${\rm dist}(\omega_{N}^{k+1},\partial \omega_{N}^{k}\setminus \partial \Omega) = \delta^{\ast}/N$. There exists $d_{\mathtt{har}}\in\mathbb{N}$ such that for each $N\in \mathbb{N}$ and $k=1,\cdots,N+1$, 
\begin{equation}
\min_{D\in \mathscr{E}_h(\omega_N^{k})}{\rm dim}\,\mathcal{\bm H}_{h}(D)\leq d_{\mathtt{har}},    
\end{equation}
where $\mathscr{E}_h(\omega_N^{k})$ is defined in \cref{eq:set_of_domains}. 
\end{assumption}

Using \cref{ass:dimension_harmonic_forms} and following along the same lines as in the proof of \cref{lem:weak_approximation}, we can prove a similar weak approximation property in this case.  
\begin{lemma}\label{lem:weak_approximation_general_domain}
Let \cref{ass:dimension_harmonic_forms} be satisfied, and let $D\subset D^{\ast}$ be from the family of domains defined in \cref{ass:dimension_harmonic_forms} with $\delta:={\rm dist} ({D},\partial D^{\ast}\setminus \partial \Omega) > 0$. Suppose that $h\leq \delta/4$. Then there exists a family of subspaces ${\bm V}_{H,\epsilon}(D)$ of ${\bm W}_{h}(D)$ with ${\rm dim} \, {\bm V}_{H,\epsilon}(D)\leq d_{\mathtt{har}}+C_{{\rm dim}}\big(|D^{\ast}|/H^{d} + (|D^{\ast}|/\delta^{d})|\log \epsilon|^{d+1}\big)$, such that for any ${\bm u}_{h}\in {\bm W}_{h}(D^{\ast})$,
\begin{align}
\begin{split}
&\inf_{{\bm z}_{h}\in {\bm V}_{H,\epsilon}(D)} \Vert {\bm u}_{h} - {\bm z}_{h}\Vert_{{\bm L}^{2}(D)}\\
&\leq C_{\rm app} \big(H(\Vert \nabla \times {\bm u}_{h}\Vert_{{\bm L}^{2}(D^{\ast})} + \delta^{-1}\Vert {\bm u}_{h}\Vert_{{\bm L}^{2}(D^{\ast})}) + \epsilon \Vert {\bm u}_{h}\Vert_{{\bm L}^{2}(D^{\ast})}\big),    
\end{split}    
\end{align}
where the constants $C_{\rm dim}$ and $C_{\rm app}$ are as in \cref{lem:weak_approximation}. 
\end{lemma}
Using \cref{lem:weak_approximation_general_domain} and similar arguments as in the proof of \cref{thm:exponential_convergence}, we can prove an exponential decay for the $n$-width, analogous to \cref{thm:exponential_convergence}. 
\begin{theorem}\label{thm:exponential_convergence_general_domain}
Let $\omega\subset \omega^{\ast}$ with $\delta^{\ast}:={\rm dist}(\omega,\partial \omega^{\ast})>0$, and let \cref{ass:dimension_harmonic_forms} and \cref{ass:POU} be satisfied. Then, there exist $n_0: = (C_1+d_{\mathtt{har}})|\omega^{\ast}|/(\delta^{\ast})^{d}$ and $b:=\frac{1}{2}\big((C_1+d_{\mathtt{har}})|\omega^{\ast}|/(\delta^{\ast})^{d}\big)^{-1/(d+1)}$, such that if $n>n_0$,
\begin{equation}
  d_{n}(P)\leq  C_2(1+\rho^{-1}) \max\{\delta^{\ast}, 1/\delta^{\ast} \}\, e^{-bn^{1/(d+1)}},  
\end{equation}
where $C_1$ and $C_2$ are as in \cref{thm:exponential_convergence}.
\end{theorem}

\section{Numerical experiments}\label{sec-5}
In this section, we present numerical results to illustrate the performance of the proposed methods. The numerical experiments are performed using the software \texttt{FreeFEM} \cite{hecht2012new}, in particular based on its \texttt{ffddm} \cite{FFD:Tournier:2019} module designed to facilitate the use of two-level Schwarz domain decomposition solvers. The local boundary value problems and the coarse problem in MS-GFEM are solved by direct solvers. The local eigenproblems are solved using the techniques presented in \cite[section 2.3]{ma2025unified} and a built-in eigensolver in \texttt{FreeFEM} which is a wrapper of \texttt{Arpack}. The convergence tolerance of GMRES corresponding to a relative residual reduction is set to $10^{-6}$. 

\begin{figure}\label{fig:model_and_solution}
    \centering
    \includegraphics[scale=0.42]{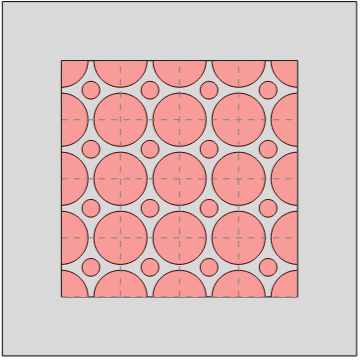}
    \hspace{4ex}
    \includegraphics[scale=0.46]{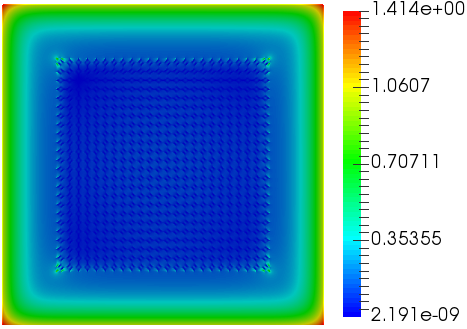}
    \caption{Left: the domain $\Omega$ containing $4\times 4$ SMC unit cells, with $\Omega_{\rm air}$ and $\Omega_{\rm SMC}$ colored in gray and red, respectively; Right: the module $ |{\bf u}_h|$ of the FE solution ${\bf u}_h$ computed with $24\times 24$ SMC unit cells in the domain and with $\sigma_{\rm air}=0.1$. We use $24\times 24$ SMC unit cells for all the computations.}
\end{figure}

\subsection{Two-dimensional example}
We first consider a 2D curl-curl problem to test the robustness of our methods for strongly heterogeneous coefficients with high contrast. The problem arises from the eddy current problem in soft
magnetic composites (SMCs) \cite{ren2019homogenization}, which is described by the following equation:
\begin{equation*}
\left\{
\begin{array}{lll}
{\displaystyle  \nabla \times ({\mu}^{-1}\nabla\times {\bf u}) + {\sigma} {\bf u} = {\bf 0}\,\;\quad {\rm in}\;\, \Omega = [-0.25, 1.25]^{2} }\\[2mm]
{\displaystyle   \;\qquad \qquad \qquad \quad {\bm n}\times {\bf u} = 1\quad \,\,{\rm on}\;\,\partial \Omega}
\end{array}
\right.
\end{equation*}
with 
\vspace{-1.5ex}
\begin{equation*}
 \mu =   \left\{
\begin{array}{lll}
{\displaystyle  50  \qquad {\rm in}\;\, \Omega_{\text{SMC}}}\\[2mm]
{\displaystyle  1\;\;\qquad {\rm in}\;\,\Omega_{\text{air}}},
\end{array}
\right.
\qquad \text{and}\quad \quad 
\sigma =   \left\{
\begin{array}{lll}
{\displaystyle  100  \quad\;\;\quad {\rm in}\;\,\Omega_{\text{SMC}}}\\[2mm]
{\displaystyle  \sigma_{\text{air}}\;\;\quad\quad {\rm in}\;\,\Omega_{\text{air}}},
\end{array}
\right.
\end{equation*}
where $\Omega_{\text{SMC}} \subset [0,1]^{2}$ and $\Omega_{\text{air}} = \Omega \setminus \Omega_{\text{SMC}}$ denote the subdomains occupied by the SMC and the air, respectively. The SMC consists of $n\times n$ copies of a scaled unit cell; see \cref{fig:model_and_solution} (left) for an illustration of the domain with $n=4$. We take $n=24$ for all our experiments. The conductivity $\sigma_{\text{air}}$ is chosen as a small positive constant, resulting in a large contrast $100/\sigma_{\rm air}$ in the coefficient $\sigma$. The computational setting is exactly the same as in \cite{ren2019homogenization}. 

The FE discretization is based on a structured triangular mesh with $h=1/800$. Unless otherwise stated, we use the lowest-order N\'{e}d\'{e}lec elements for the discretization, with about $1.9\times 10^{6}$ unknowns in the resulting linear system. The overlapping decomposition $\{\omega_i\}$ of the domain $\Omega$ is defined by first partitioning $\Omega$ into $m\times m$ uniform non-overlapping subdomains, and then enlarging each subdomain by adding two layers of adjoining fine mesh elements. Each $\omega_i$ is further extended by adding several layers of fine mesh elements to create the corresponding oversampling domain $\omega_i^{\ast}$. The number of these additional layers is denoted by `Ovsp', and thus the actual oversampling size is $\text{Ovsp} \times h$. We denote by $n_{\rm loc}$ the number of local eigenvectors used per subdomain for building the coarse space. We use a partition of unity similar to the one described in \cite{jolivet2014overlapping}, but with a distance function defined on edges instead of on vertices.

\begin{figure}\label{fig:eigenvalue_and_error}
    \centering
    \includegraphics[scale=0.40]{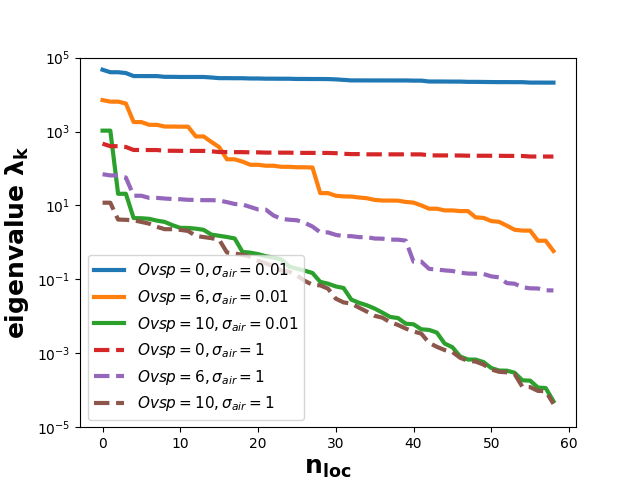}
    \hspace{-2ex}
    \includegraphics[scale=0.40]{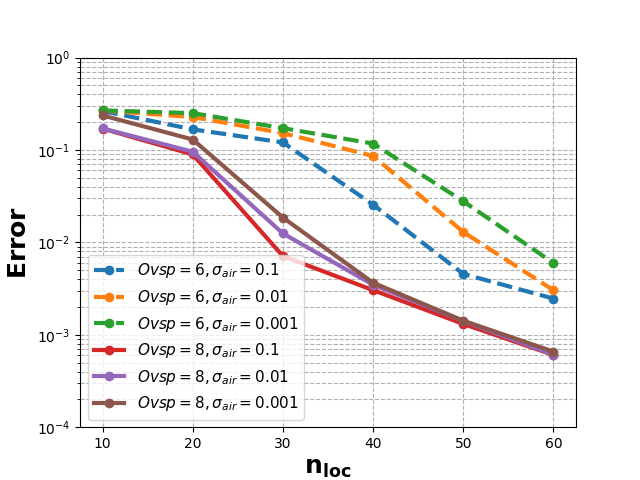}
    \caption{2D example: the computed eigenvalues of MS-GFEM on an interior subdomain with $m=12$ (left) and the (relative) errors of MS-GFEM with $m=8$ (right).}
\end{figure}

\Cref{fig:eigenvalue_and_error} (left) plots the computed eigenvalues of MS-GFEM on an interior subdomain with different oversampling sizes. Without oversampling ($\text{Ovsp}=0$), the eigenvalues are nearly constant, but when a small oversampling is used, they decay exponentially as predicted. Moreover, a larger oversampling size yields a faster eigenvalue decay. We also observe that the eigenvalues corresponding to $\sigma_{\rm air} = 0.01$ are noticeably larger than those corresponding to $\sigma_{\rm air} = 1$ for very small oversampling sizes ($\text{Ovsp}=0$ or 6), but are approximately the same for the size $\text{Ovsp}=10$ (corresponding to an oversampling ratio of 1.28). This shows that the effect of a high coefficient contrast on the eigenvalues can be offset by using a moderate oversampling. \Cref{fig:eigenvalue_and_error} (right) displays the convergence of MS-GFEM with respect to $n_{\rm loc}$ for different choices of Ovsp and $\sigma_{\rm air}$. We see that with $\text{Ovsp}=6$, the errors are noticeably larger for higher coefficient contrasts, but with a slightly larger oversampling size $\text{Ovsp}=8$, they are comparable and are significantly smaller than in the former case. This shows that using a slightly larger oversampling can significantly improve the robustness of the method for coefficient contrasts. 

\begin{figure}\label{fig:residual}
    \centering
    \includegraphics[scale=0.39]{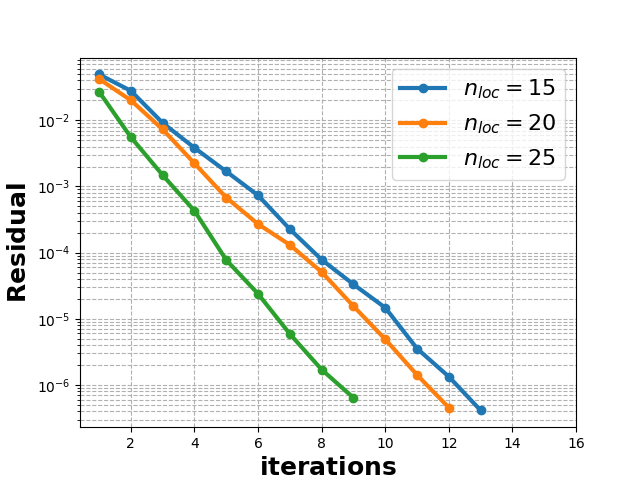}
    \includegraphics[scale=0.39]{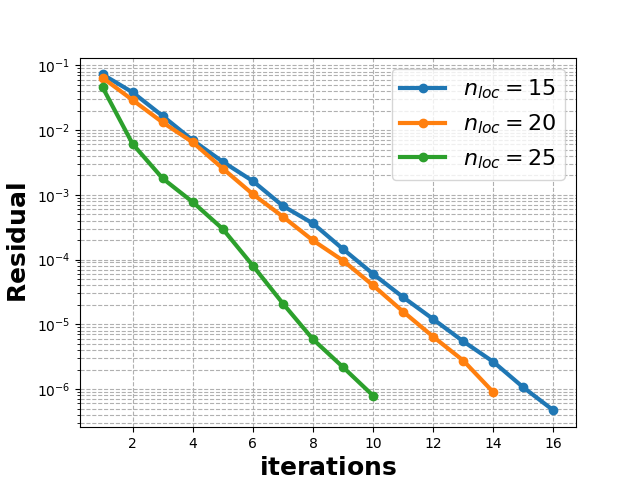}
    \caption{2D example: the convergence history of the preconditioned GMRES for $\sigma_{\rm air} =1$ (left) and  $\sigma_{\rm air} =0.01$ (right) with $m=8$ and $\text{Ovsp}=6$.}
\end{figure}

\begin{table}\label{table:iteration_count}
\centering
\begin{tabular}{lllllll}
\toprule
$\text{Ovsp}\textbackslash n_{\rm loc} $& 10 & 20 & 30 & 40 & 50 & 60 \\
\midrule
2 & 21 (32) & 21 (31) & 21 (29) & 21 (24) & 21 (23) & 19 (22)\\
4 & 18 (29) & 18 (20) & 17 (19) & 14 (16) & 13 (16) & 11 (14) \\
6 & 16 (18) & 15 (17) & 11 (15) & 9 (12) & 8 (9) & 3 (4)\\
8 & 14 (16) & 9 (11) & 5 (7) & 3 (4) & 2 (3) & 2 (2)\\
10 & 6 (13) & 3 (7) & 2 (3) & 2 (3) & 2  (2) & 2 (2)\\
\bottomrule
\end{tabular}
\caption{2D example: the iteration counts of GMRES for $\sigma_{\rm air} = 0.001$ with $m = 8$. The corresponding iteration numbers for the quadratic N\'{e}d\'{e}lec elements are shown in the parentheses.}
\end{table}

\Cref{fig:residual} depicts the convergence history of GMRES preconditioned with MS-GFEM with different local space sizes $n_{\rm loc}$. We observe that a larger $n_{\rm loc}$ leads to a faster residual reduction as expected. Moreover, there is only a small difference in the iteration counts between $\sigma_{\rm air} =1$ and $\sigma_{\rm air} = 0.01$. Notably, the algorithm can converge in 16 iterations with $n_{\rm loc}=15$ and $m=8$ -- the size of the coarse problem is as small as 1/2000 of the fine-scale problem. \Cref{table:iteration_count} exhibits the iteration counts of the preconditioned GMRES with different oversampling and local space sizes for $\sigma_{\rm air} =0.001$. We observe that with a larger oversampling size, the decrease in the iteration counts with increasing $n_{\rm loc}$ is more significant. Moreover, with reasonably large Ovsp and $n_{\rm loc}$, the algorithm can converge in only two iterations. However, using large oversampling and local space sizes results in significant computational cost, particularly in the solution of the local eigenproblems. In fact, as shown in \cite{strehlow2024fast}, this is generally not optimal for the overall computational efficiency of the algorithm. We also see from \cref{table:iteration_count} that the algorithm for the quadratic N\'{e}d\'{e}lec elements generally needs more iterations than that for the lowest-order N\'{e}d\'{e}lec elements, but for large local space sizes, the iteration counts are almost the same.


\begin{figure}\label{fig:3D_model}
    \centering
    \includegraphics[scale=0.18]{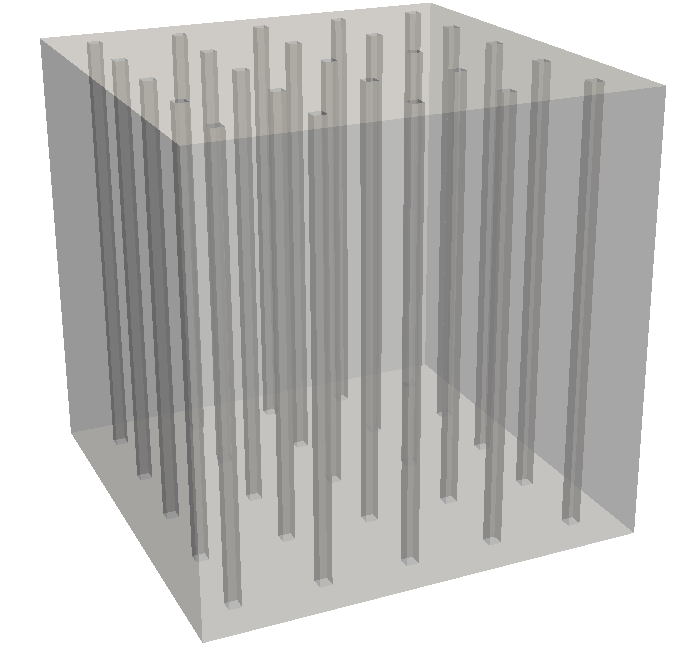}
    \hspace{4ex}
      \includegraphics[scale=0.13]{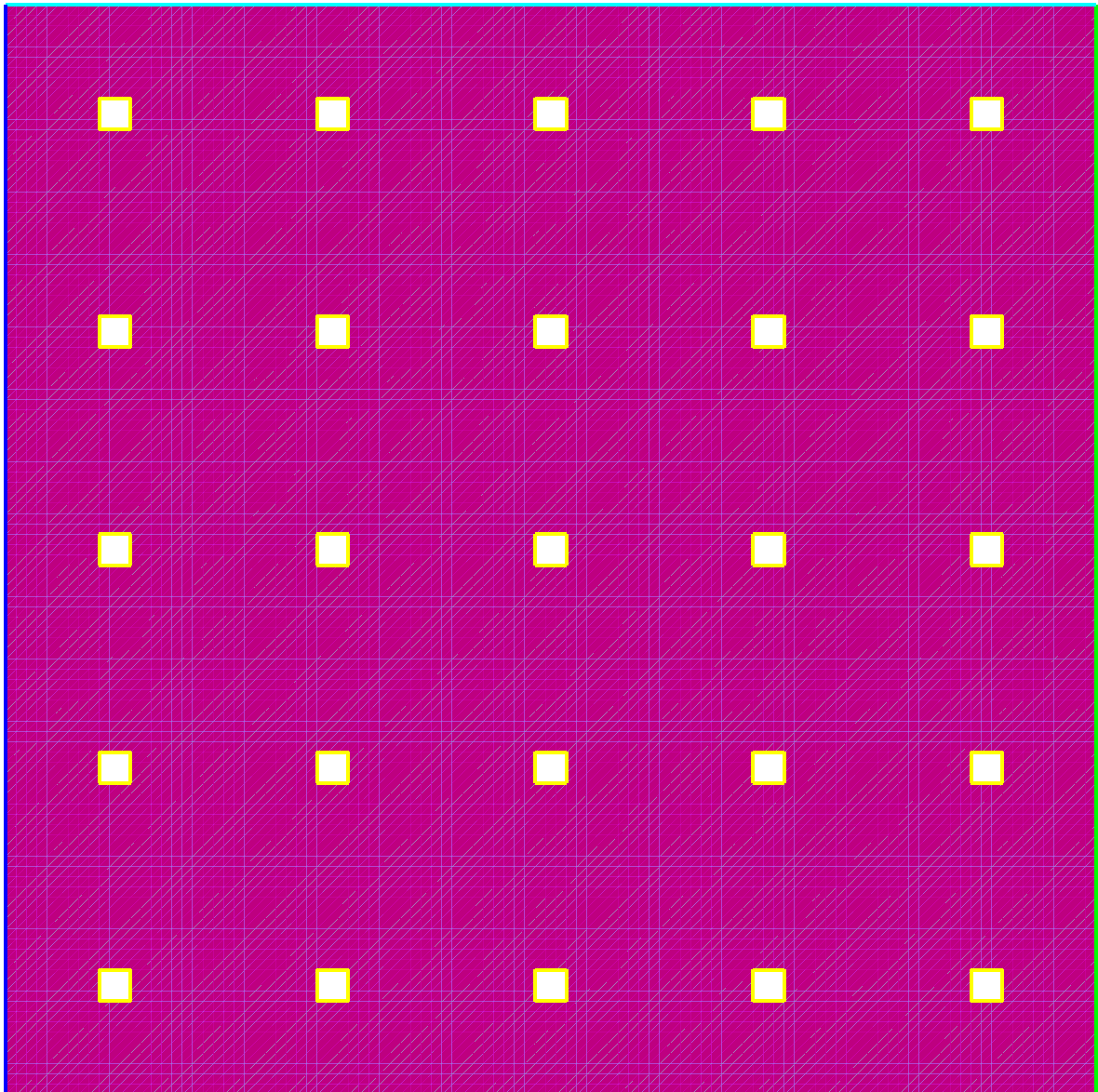}

      \vspace{4ex}
        \includegraphics[scale=0.13]{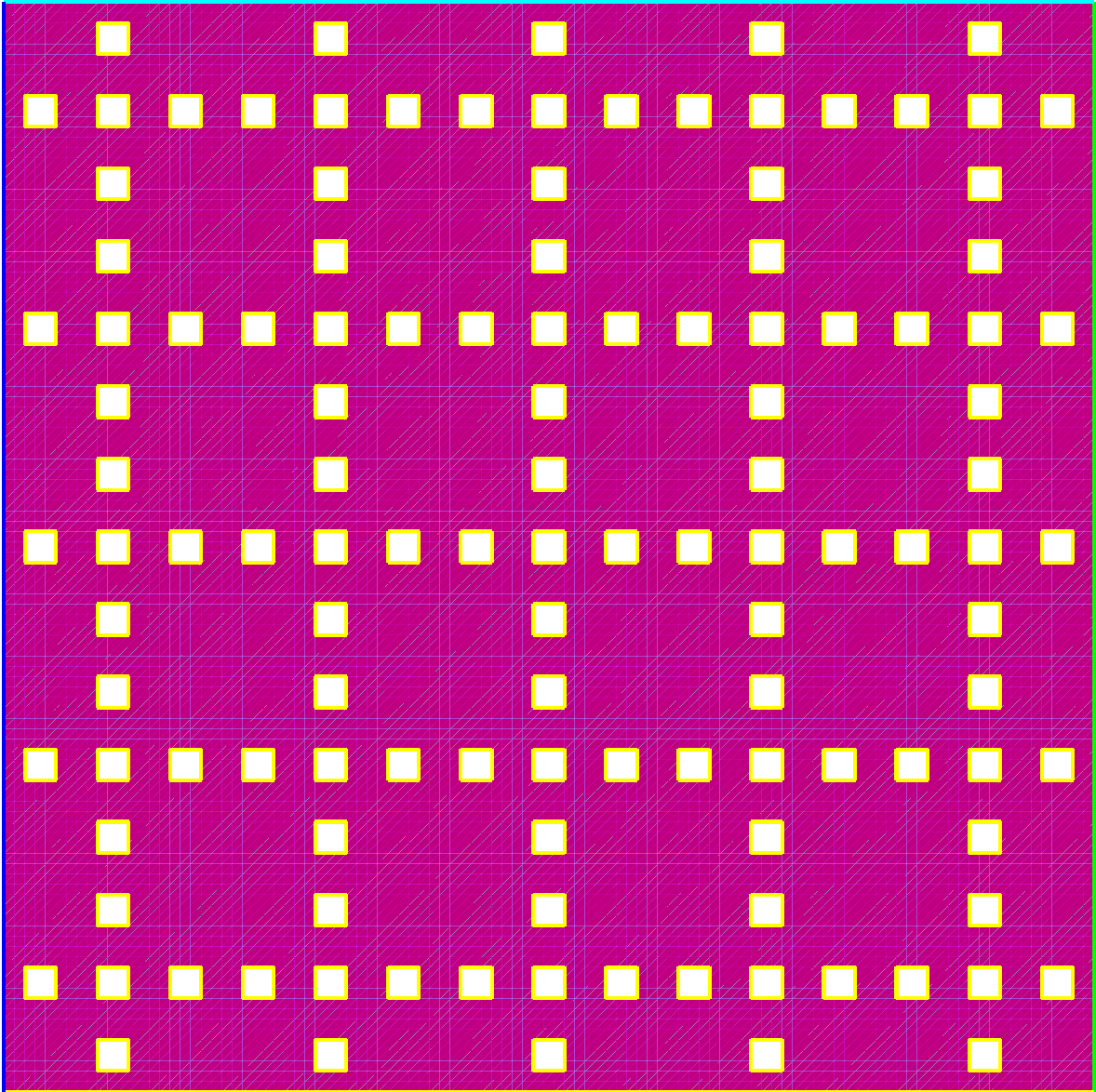}
        \hspace{4ex}
          \includegraphics[scale=0.13]{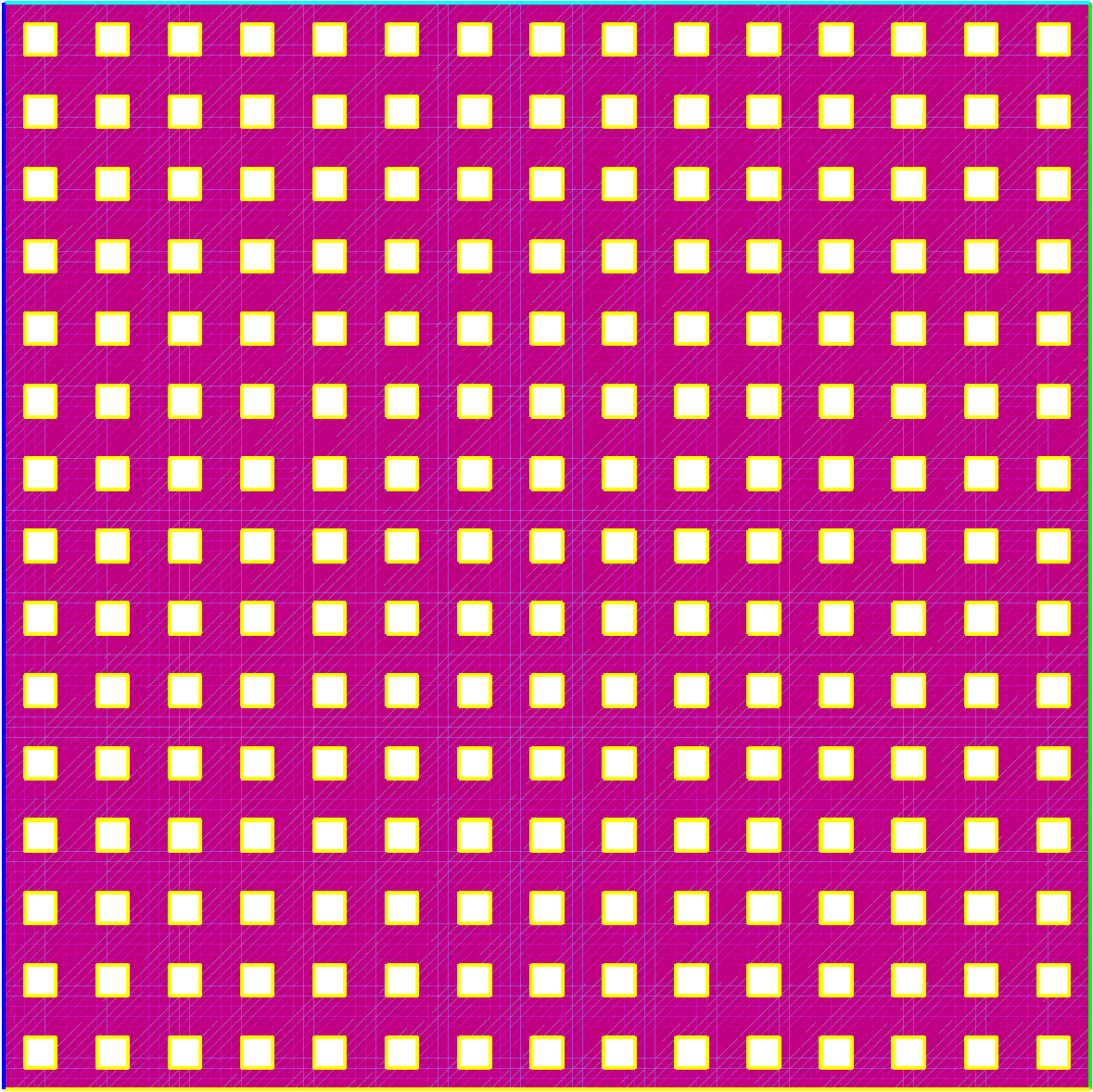}
    \caption{Computational domains with an increasing number of holes that extend along the $z$-direction: 25 (top right), 125 (bottom left), and 225 (bottom right) holes, all shown in the $x$-$y$ plane. The 3D diagram of the domain with 25 holes is shown on the top left. }
\end{figure}

\subsection{Three-dimensional example}\label{3D example}
In this example, we test the robustness of our methods for domains with non-trivial topologies. We choose the constant coefficients $\kappa = \nu =1$ and the source term ${\bm f} = (1, 1, 1)$ in equation \cref{eq:continuous_model_prob} for our test. The computational domain $\Omega$ is a unit cube with holes that extend along the $z$-direction of the domain; see \cref{fig:3D_model} for three domain configurations. In particular, we refer to the models with holes $25$, $125$, and $225$ as model 1 (`m1'), model 2 (`m2'), and model 3 (`m3'), respectively. The FE discretizations for these models are based on the same structured tetrahedral mesh with $h=1/100$ and on the lowest-order N\'{e}d\'{e}lec elements. The resulting linear systems contain 6165240, 5795140, 5425040 unknowns, respectively. 

To implement MS-GFEM, we first partition the computational domain into $250$ uniform non-overlapping subdomains, with $5\times 5$ subdomains in the $xy$-plane and 10 subdomains in the $z$-direction. The numbers of holes in each subdomain in the three models are 1, 5, and 9, respectively. Each subdomain is extended by adding one layer of fine mesh elements to define the overlapping decomposition $\{\omega_i\}$. As before, we denote by `Ovsp' the number of layers of fine mesh elements added to each $\omega_i$ to create the corresponding oversampling domain $\omega_i^{\ast}$, and by $n_{\rm loc}$ the number of local eigenvectors used per subdomain for building the coarse space. We note that due to the presence of holes, the oversampling domains can include boundaries with a "sawtooth" shape. We use the same partition of unity as for the 2D example.  

\Cref{fig:3D_eigenvalue_decay} (left) depicts the computed eigenvalues of MS-GFEM on an interior subdomain for the three models with different oversampling sizes. We observe that for all three models, the eigenvalues decay exponentially with a higher rate for a larger oversampling size, which agrees well with our theoretical predictions. Moreover, with more holes in the subdomain, the corresponding eigenvalues $\lambda_{k}$ are significantly smaller for $k>20$, but with a similar decay rate. Interestingly, the topology of the subdomain significantly affects the first few eigenvalues, as shown in \cref{fig:3D_eigenvalue_decay} (right). In particular, we observe that for the three models, the number of the first eigenvalues that remain nearly constant corresponds to the number of holes in the subdomain $\omega_i$. This demonstrates that the eigenvalues of MS-GFEM indeed encode important topological information about the associated subdomain.

\begin{figure}\label{fig:3D_eigenvalue_decay}
    \centering
    \includegraphics[scale=0.40]{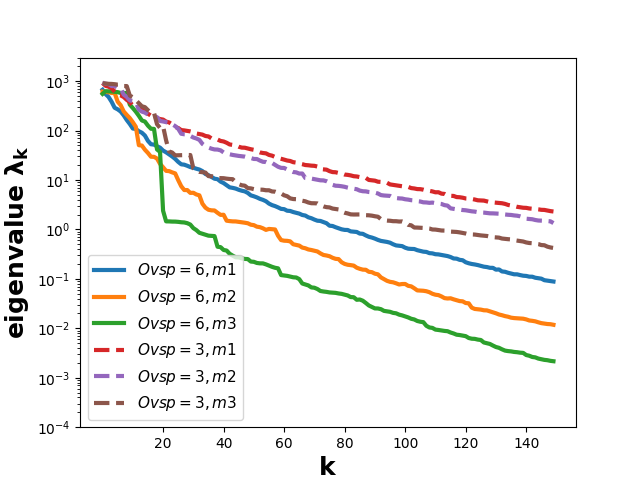}
    \hspace{-3ex}
    \includegraphics[scale=0.4]{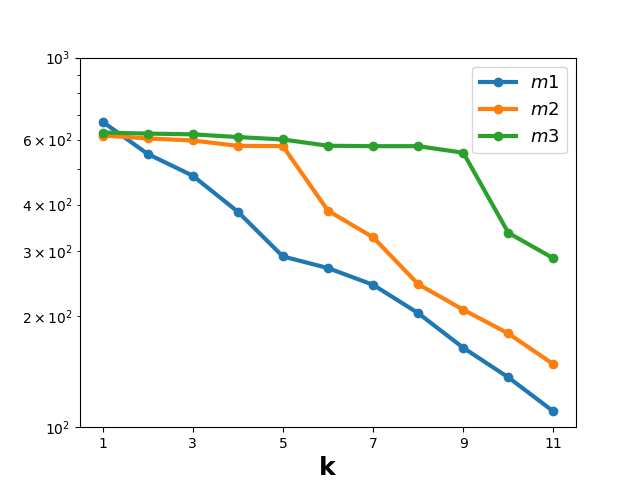}
    \caption{3D example: the computed eigenvalues of MS-GFEM on an interior subdomain for the three models (left) and the first 11 eigenvalues with $\text{Ovsp}=6$ (right). }
\end{figure}

\begin{figure}\label{fig:3D_error}
    \centering
    \includegraphics[scale=0.50]{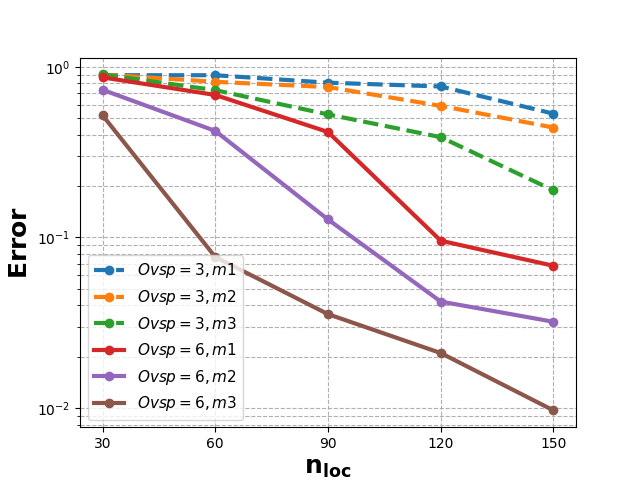}
    \caption{3D example: the (relative) errors of MS-GFEM.}
\end{figure}

\Cref{fig:3D_error} displays the convergence of MS-GFEM with increasing local space sizes $n_{\rm loc}$. We observe that as for the 2D example, using a slightly larger oversampling significantly improves the convergence rate of the method. But as predicted by our theory, the convergence rate in 3D is much lower than that in 2D. In addition, the more holes in the domain, the smaller the errors are, which is in line with the behavior of the eigenvalues shown above. \Cref{tab:3D_iteration_count} gives the iteration counts of the preconditioned GMRES with different local space sizes. We observe that increasing the local space size leads to only a limited decrease in the iteration counts, especially for a very small oversampling size. Although increasing the oversampling and local space sizes can reduce the iteration numbers needed, it also results in a significant growth in computational overhead. In practice, it is often preferable to use small oversampling and local space sizes with a few more iterations. The advantage of the iterative solver is that even with a fairly small coarse space, the method can converge reasonably fast. Indeed, as shown in the table, the method can converge in 34 iterations for all three models with $\text{Ovsp}=2$ and $n_{\rm loc} = 20$. With this setting, the size of the coarse problem is only 2500, achieving (at least) a 2170$\times$ reduction in model size.

\begin{table}[h]\label{tab:3D_iteration_count}
\centering
\begin{tabular}{l|ll|ll|ll}
\toprule
 \multirow{2}{*}{$n_{\rm loc}$} &\multicolumn{2}{c}{\text{model} 1} & \multicolumn{2}{c}{\text{model} 2} &  \multicolumn{2}{c}{\text{model} 3} \\
 &$\text{Ovsp}=2$ & $\text{Ovsp}=4$ & $\text{Ovsp}=2$ & $\text{Ovsp}=4$& $\text{Ovsp}=2$ & $\text{Ovsp}=4$ \\
\midrule
 20 & \quad 34  & \quad 33  & \quad 33 & \quad 21 & \quad 32  & \quad 23 \\
 30 & \quad 33 & \quad 21  & \quad 24 & \quad 20 & \quad 20 & \quad 15 \\
 40 & \quad 30  & \quad 19 & \quad 22 & \quad 18 & \quad 20 & \quad 13 \\
 50 & \quad 23 & \quad 17 & \quad 22 & \quad 18 & \quad 20 & \quad 13 \\
 60 & \quad 22  & \quad 16  & \quad 21 & \quad 16 & \quad 20 & \quad 12  \\
 80 & \quad 21 & \quad 14  & \quad 21 & \quad 13 & \quad 19 & \quad 8 \\
 100 & \quad 21  & \quad 11 & \quad 20 & \quad 13 & \quad 19 & \quad 6  \\
 120 & \quad 20 & \quad 12 & \quad 19 & \quad 10 & \quad 15 & \quad 5 \\
\bottomrule
\end{tabular}
\caption{3D example: the iteration counts of GMRES for the three models.}
\end{table}

\section{Conclusions}

We have applied the discrete MS-GFEM as a model reduction technique for discretized \( H(\mathrm{curl}) \) elliptic problems, and have rigorously established its exponential convergence under minimal assumptions. To address the computational challenge of solving large coarse-scale systems, we further reformulated the discrete MS-GFEM as a two-level restricted additive Schwarz method, with provable convergence guarantees. Numerical experiments demonstrate the effectiveness of the proposed methods in handling problems with highly heterogeneous, high-contrast coefficients and topologically complex domains. Notably, the iterative solver achieves convergence within a small number of iterations, even when the coarse problem size is as small as \(1/2000\) of the fine-scale problem. Future work will explore the extension of these techniques to time-harmonic Maxwell's equations, by combining the methodology presented in this paper with the preconditioning strategy introduced in~\cite{chen2007adaptive}.



\bibliographystyle{siamplain}
\bibliography{lit}
\end{document}

%% file: AharmonicGenEOShared.tex
\usepackage{amsfonts}
\usepackage{graphicx}
\usepackage{epstopdf}
\usepackage{algorithmic}
\usepackage{nicefrac}
\usepackage{textgreek}
\usepackage{stmaryrd}
\ifpdf
  \DeclareGraphicsExtensions{.eps,.pdf,.png,.jpg}
\else
  \DeclareGraphicsExtensions{.eps}
\fi
\usepackage{booktabs}

\usepackage{multirow,bigstrut}
\usepackage{amsmath,amsfonts,amssymb}
\usepackage{mathrsfs}
\usepackage{color}
\usepackage{upgreek}
\usepackage{bm}
\usepackage{verbatim}
\usepackage{mathtools}
\usepackage{calligra}
\usepackage[T1]{fontenc}
\usepackage{subfigure}
\usepackage{pgfplots}
\usepackage{xcolor}

\usepackage{tikz}
  \usetikzlibrary{calc,shapes,decorations,decorations.text, mindmap,shadings,patterns,matrix,arrows,intersections,automata,backgrounds}
  
\usetikzlibrary{shapes,arrows}

\usepackage{verbatim}
\usepackage{exscale}
\usepackage{relsize}
\usepackage{enumitem}
\theoremstyle{plain}

\newsiamremark{remark}{Remark}
\newsiamremark{hypothesis}{Hypothesis}
\crefname{hypothesis}{Hypothesis}{Hypotheses}
\newsiamthm{claim}{Claim}
\newtheorem{assumption}[theorem]{Assumption}

\title{Multiscale model reduction and two-level Schwarz preconditioner for H(curl) elliptic problems
}

\author{Chupeng Ma\thanks 
    {School of Sciences, Great Bay University, Dongguan 523000, China; Great Bay Institute for Advanced Study, Dongguan 523000, China (\email{chupeng.ma@gbu.edu.cn}).}
\and Yongwei Zhang\thanks 
    {School of Mathematics and Statistics, Zhengzhou University, Zhengzhou 450001, China (\email{ywzhang@zzu.edu.cn}).}    
}

\usepackage{amsopn}
